\documentclass[a4paper, 11pt]{article}
\usepackage[utf8]{inputenc}
\usepackage{amsmath, amsfonts, amssymb}
\usepackage{mathrsfs, bm}
\usepackage[dvipsnames]{xcolor}
\usepackage{listings}
\usepackage{verbatim,mathtools,graphicx}
\usepackage[T1]{fontenc}
\usepackage{newpxtext}
\usepackage[varbb]{newpxmath}
\usepackage{tikz, tikz-cd}
\usepackage{enumitem,xparse,url}
\usepackage[most]{tcolorbox}
\usepackage{multirow,array}
\usepackage{hyperref,bookmark}
\usepackage{stmaryrd}

\newcolumntype{x}[1]{>{\centering\let\newline\\\arraybackslash\hspace{0pt}}p{#1}}

\newenvironment{diagram*}{\vspace{0.1cm}\begin{center}\begin{tikzcd}}{\end{tikzcd}\end{center}\vspace{0.1cm}}

\tcbuselibrary{theorems}
\tcbuselibrary{skins}
\tcbuselibrary{breakable}
\tcbuselibrary{vignette}

\tcbset{
  plategeneric/.style={
    enhanced, breakable,
    borderline west={2pt}{0pt}{white!75!black},
    colback=white,
    top=1.3pt,
    bottom=1.3pt,
    left=7pt,
    right=7pt,
    grow to left by=10pt,
    grow to right by=10pt,
    boxrule=0pt,
    sharp corners
  },
  plate/.style={
    plategeneric,
    frame hidden,
    beforeafter skip balanced=0.22\baselineskip plus 1pt,
    height fixed for=first and middle,
    lines before break=1,
    pad at break*=1.1mm,
    detach title,
    coltitle=black
  },
  fancythm/.style={
    plate,
    top=4pt,
    bottom=4pt,
    borderline west={2pt}{0pt}{Goldenrod!70!black},
    colback=Goldenrod!5!white,
    before upper={\parindent15pt\noindent\textup{\textbf{\tcbtitle.}}\,\,},
    fontupper=\itshape
  },
  fancyconj/.style={
    plate,
    borderline west={2pt}{0pt}{Goldenrod!80!black},
    before upper={\parindent15pt\noindent\textup{\textbf{\tcbtitle.}}\,\,},
    fontupper=\itshape
  },
  fancyproof/.style={
    plate,
    borderline west={2pt}{0pt}{Peach!75!black},
    colback=Peach!1!white,
    fonttitle=\itshape,
    before upper={\parindent15pt\noindent\tcbtitle\textit{.}\,\,\,},
    fontupper=\upshape
  },
  fancydef/.style={
    plate,
    borderline west={2pt}{0pt}{NavyBlue!85!black},
    colback=white,
    fonttitle=\bfseries,
    before upper={\parindent15pt\noindent\tcbtitle\textbf{.}\,\,\,}
  },
  fancynotation/.style={
    plate,
    borderline west={2pt}{0pt}{Cyan!65!black},
    colback=white,
    fonttitle=\bfseries,
    before upper={\parindent15pt\noindent\tcbtitle\textbf{.}\,\,\,}
  },
  fancyexample/.style={
    plate,
    borderline west={2pt}{0pt}{Plum!75!black},
    colback=white,
    fonttitle=\bfseries,
    before upper={\parindent15pt\noindent\tcbtitle\textbf{.}\,\,\,}
  },
  fancyremark/.style={
    plate,
    borderline west={2pt}{0pt}{Emerald!75!black},
    colback=white,
    fonttitle=\itshape,
    before upper={\parindent15pt\noindent\tcbtitle\textit{.}\,\,\,}
  },
  fancyexercise/.style={
    plate,
    borderline west={2pt}{0pt}{WildStrawberry!85!black},
    colback=WildStrawberry!10!white,
    fonttitle=\bfseries,
    before upper={\parindent15pt\noindent\tcbtitle\textbf{.}\,\,\,}
  },
  fancysoln/.style={
    plate,
    borderline west={2pt}{0pt}{YellowGreen!85!black},
    colback=white,
    fonttitle=\bfseries,
    before upper={\parindent15pt\noindent\tcbtitle\textbf{.}\,\,\,}
  },
  fancywarning/.style={
    plate,
    borderline west={2pt}{0pt}{Red!75!black},
    colback=Red!15!white,
    before upper={\parindent15pt\noindent\tcbtitle\textbf{.}\,\,\,},
    fonttitle=\bfseries,
    fontupper=\upshape
  },
  fancynote/.style={
    plategeneric,
    grow to right by=4pt,
    borderline west={2pt}{0pt}{ProcessBlue!75!black},
    colback=ProcessBlue!15!white,
    colbacktitle=ProcessBlue!50!white,
    coltitle={ProcessBlue!30!black},
    title={Note},
    fonttitle=\bfseries,
    fontupper=\upshape,
    before upper={\parindent15pt\noindent}
  },
  fancytodo/.style={
    plategeneric,
    grow to right by=4pt,
    borderline west={2pt}{0pt}{Green!75!black},
    colback=Green!15!white,
    colbacktitle=Green!50!white,
    coltitle={Green!40!black},
    title={TO-DO},
    fonttitle=\bfseries,
    fontupper=\upshape,
    before upper={\parindent15pt\noindent}
  }
}

\newcommand{\ThmTitleNoOptional}[3]{#1 #2.#3}
\newcommand{\ThmTitleOptional}[4]{#1 #2.#3: #4}
\NewDocumentCommand\goodnewtcbtheorem{ommmm}{%
  \IfNoValueTF{#1}{%
  \newenvironment{#2}[1][]{\refstepcounter{#5}\begin{tcolorbox}[title={\if\relax\detokenize{##1}\relax{\ThmTitleNoOptional{#3}{\thesection}{\arabic{#5}}}\else{\ThmTitleOptional{#3}{\thesection}{\arabic{#5}}{##1}}\fi},#4]}{\end{tcolorbox}}%
  \newenvironment{#2*}[1][]{\begin{tcolorbox}[title={\if\relax\detokenize{##1}\relax{#3}\else{#3: ##1}\fi},#4]}{\end{tcolorbox}}%
  }{%
  \newenvironment{#2}[1][]{\refstepcounter{#5}\begin{tcolorbox}[title={\if\relax\detokenize{##1}\relax{#3 #1\arabic{#5}}\else{#3 #1\arabic{#5}: ##1}\fi},#4]}{\end{tcolorbox}}%
  \newenvironment{#2*}[1][]{\begin{tcolorbox}[title={\if\relax\detokenize{##1}\relax{#3}\else{#3: ##1}\fi},#4]}{\end{tcolorbox}}%
  }
}
\NewDocumentCommand\goodrenewtcbtheorem{ommmm}{%
  \IfNoValueTF{#1}{%
  \renewenvironment{#2}[1][]{\refstepcounter{#5}\begin{tcolorbox}[title={\if\relax\detokenize{##1}\relax{#3 \thesection.\arabic{#5}}\else{#3 \thesection.\arabic{#5}: ##1}\fi},#4]}{\end{tcolorbox}}%
  \renewenvironment{#2*}[1][]{\begin{tcolorbox}[title={\if{#3}\else{#3: ##1}\fi},#4]}{\end{tcolorbox}}%
  }{%
  \renewenvironment{#2}[1][]{\refstepcounter{#5}\begin{tcolorbox}[title={\if\relax\detokenize{##1}\relax{#3 #1\arabic{#5}}\else{#3 #1\arabic{#5}: ##1}\fi},#4]}{\end{tcolorbox}}%
  \renewenvironment{#2*}[1][]{\begin{tcolorbox}[title={\if{#3}\else{#3: ##1}\fi},#4]}{\end{tcolorbox}}%
  }
}

\numberwithin{thmcounter}{section}

\goodnewtcbtheorem{theorem}{Theorem}{fancythm}{thmcounter}
\goodnewtcbtheorem{theoremdef}{Theorem-Definition}{fancythm}{thmcounter}
\goodnewtcbtheorem{proposition}{Proposition}{fancythm}{thmcounter}
\goodnewtcbtheorem{propositiondef}{Proposition-Definition}{fancythm}{thmcounter}
\goodnewtcbtheorem{corollary}{Corollary}{fancythm}{thmcounter}
\goodnewtcbtheorem{lemma}{Lemma}{fancythm}{thmcounter}
\goodnewtcbtheorem{fact}{Fact}{fancythm}{thmcounter}

\goodnewtcbtheorem{conjecture}{Conjecture}{fancyconj, fontupper=\upshape}{thmcounter}

\goodnewtcbtheorem{definition}{Definition}{fancydef}{thmcounter}

\goodnewtcbtheorem{notation}{Notation}{fancynotation}{thmcounter}
\goodnewtcbtheorem{terminology}{Terminology}{fancynotation}{thmcounter}
\goodnewtcbtheorem{convention}{Convention}{fancynotation}{thmcounter}
\goodnewtcbtheorem{question}{Question}{fancynotation}{thmcounter}
\goodnewtcbtheorem{axiom}{Axiom}{fancynotation}{thmcounter}

\goodnewtcbtheorem{example}{Example}{fancyexample}{thmcounter}
\goodnewtcbtheorem{construction}{Construction}{fancyexample}{thmcounter}
\goodnewtcbtheorem{sketch}{Sketch}{fancyexample}{thmcounter}

\goodnewtcbtheorem{remark}{Remark}{fancyremark}{thmcounter}
\goodnewtcbtheorem{discussion}{Discussion}{fancyremark}{thmcounter}

\goodnewtcbtheorem[]{exercise}{Exercise}{fancyexercise}{excounter}
\goodnewtcbtheorem[]{task}{Task}{fancyexercise}{taskcounter}
\goodnewtcbtheorem[]{problem}{Problem}{fancyexercise}{probcounter}

\goodnewtcbtheorem[]{solution}{Solution}{fancysoln}{solcounter}

\goodnewtcbtheorem{warning}{Warning}{fancywarning}{thmcounter}

\newtcolorbox{note}{fancynote}
\newtcolorbox{todo}{fancytodo}

\newcommand{\qedsymbol}{$\blacksquare$}
\newtcolorbox{proof}[1][Proof]{fancyproof, title={#1}}
\AtEndEnvironment{proof}{\null\hfill\qedsymbol}

\renewcommand\qedsymbol{$\blacksquare$}
\makeatletter
\newcommand{\oset}[3][0ex]{%
  \mathrel{\mathop{#3}\limits^{
    \vbox to#1{\kern-2\ex@
    \hbox{$\scriptstyle#2$}\vss}}}}
\makeatother

\newcommand{\bigmid}{\biggm|}

\newcommand{\To}{\Rightarrow}

\newcommand{\bfG}{\ensuremath{\mathbf{G}}}

\newcommand{\sfB}{\ensuremath{\mathsf{B}}}
\newcommand{\sfC}{\ensuremath{\mathsf{C}}}

\newcommand{\bbJ}{\ensuremath{\mathbb{J}}}

\newcommand{\calA}{\ensuremath{\mathcal{A}}}

\newcommand{\calE}{\ensuremath{\mathcal{E}}}

\newcommand{\calT}{\ensuremath{\mathcal{T}}}

\newcommand{\frakC}{\ensuremath{\mathfrak{C}}}

\newcommand{\frakF}{\ensuremath{\mathfrak{F}}}

\newcommand{\frakN}{\ensuremath{\mathfrak{N}}}

\newcommand{\frakS}{\ensuremath{\mathfrak{S}}}

\newcommand{\R}{\mathbb{R}}
\newcommand{\C}{\mathbb{C}}

\newcommand{\N}{\mathbb{N}}

\newcommand{\adj}{\operatorname{adj}}

\newcommand{\img}{\operatorname{im}}
\newcommand{\End}{\operatorname{End}}
\newcommand{\Aut}{\operatorname{Aut}}
\newcommand{\coker}{\operatorname{coker}}

\newcommand{\Span}{\operatorname{Span}}

\newcommand{\iso}{\xrightarrow{\sim}}

\newcommand{\sur}{\twoheadrightarrow}

\newcommand{\HH}{\operatorname{H}}

\newcommand{\pre}{\mathrm{pre}}

\newcommand{\id}{\mathrm{id}}

\usepackage[textwidth=15.5cm,top=2.5cm,bottom=2.5cm]{geometry}
\usepackage{appendix}

\title{Conjugating by singular operators \\ \large On the boundedness of similarity transforms near singular points}
\author{Daniel Falkowski \& Carl-Fredrik Lidgren}

\newcommand{\1}{\ensuremath{\mathbb{1}}}

\newcommand{\pw}{\ensuremath{\mathrm{pw}}}

\usepackage[backend=biber, style=alphabetic]{biblatex}
\begin{document}
\maketitle

\begin{abstract}
	We consider the question of, given operators \(A,Z\) and a sequence of invertible operators \(U_n\to Z\), whether the sequence \(U_nAU_n^{-1}\) is bounded in norm, as well as generalizations of this
	where \(U_nAU_n^{-1}\) is modified by some bounded linear map on bounded linear operators. In the setting of Hilbert spaces, we provide a complete classification in terms of algebraic criteria of
	those \(A\) for which such a sequence exists, as long as \(Z\) is of \emph{generalized index zero,} which always holds in finite-dimensional contexts. In the process, we prove that particular coefficients
	arising in inverses of certain \emph{good paths} going to \(Z\) can also be classified in terms of an entirely algebraic criterion.
\end{abstract}

\tableofcontents

\clearpage
\section{Introduction}\label{section:introduction}
\subsection{The Problem \& the Approach}\label{subsection:problem-statement}
This document deals with trying to answer the following basic question, which comes up naturally in the context of jointly diagonalizing matrices (see e.g.\ \cite{troedsson2024jointeigendecompositionmatrices}):
\begin{question}\label{question:pathwise-question}
	Let \(U_k\to Z\) be a sequence of invertible matrices converging to a singular matrix \(Z\), and let \(A\) be a square matrix. Does the sequence
	\[ \| U_kAU_k^{-1} \| \]
	diverge to infinity? What if one modifies the sequence by, for example, deleting the diagonal of \(U_kAU_k^{-1}\)?
\end{question}

This question is in general hard to answer, as the behaviour of a given ``path'' \(U_k\to Z\) may be highly particular to the specific path itself. Therefore, as
a simplifying approximation, we may wish to instead ask about all approach paths at once. Doing this naturally introduces two extremes: first, asking for which
pairs \((A,Z)\) it holds that for \emph{every} sequence \(U_k\to Z\), the sequence as in Question \ref{question:pathwise-question} does not diverge to infinity, which can be rephrased as

\begin{question}\label{question:the-question}
	Let \(V=\C^n\). Determine for which \(Z\in\End(V)\) and \(A\in\End(V)\) it holds that
	\[ \limsup_{r\to 0}\sup_{\substack{U\in\Aut(V) \\ \|U-Z\|<R}}\|UAU^{-1}\| = \limsup_{r\to 0}\sup_{\substack{U\in\Aut(V) \\ \|U-Z\|<r}}\sup_{w\not= 0}\frac{\|UAw\|}{\|Uw\|} < \infty. \]
\end{question}

Secondly, asking for which pairs \((A,Z)\) there exists \emph{some} sequence \(U_k\to Z\) such that the sequence in Question \ref{question:pathwise-question} does not diverge to infinity, which
may be rephrased as
\begin{question}\label{question:local-infimum-question}
	Let \(Z\) be singular, and let \(A\) be arbitrary. Determine when
	\[ \limsup_{r\to 0}\inf_{\substack{U\in\Aut(\C^n) \\ \|U - Z\| < r}}\| UAU^{-1} \| < \infty. \]
\end{question}

Taking into account that we may also wish to ``modify'' \(U_kAU_k^{-1}\) by some operation, such as deleting the diagonal as suggested in Question \ref{question:pathwise-question}, there is a modification
one can make to the above questions to facilitate that. In particular,
\begin{question}\label{question:hadamard-question}
	Let \(\varphi\!:\End(\C^n)\to\End(\C^n)\) be a bounded linear map. Determine for which \(A,Z\in\End(\C^n)\) we have
	\[ \limsup_{r\to 0}\sup_{\substack{U\in\Aut(\C^n), \\ \|U-Z\|<r}}\|\varphi(UAU^{-1})\| < \infty.\]
\end{question}
\begin{question}\label{question:local-hadamard-infimum-question}
	Let \(\varphi\!:\End(\C^n)\to\End(\C^n)\) be a bounded linear map. Determine for which \(A,Z\in\End(\C^n)\) we have
	\[ \limsup_{r\to 0}\inf_{\substack{U\in\Aut(\C^n) \\ \|U - Z\| < r}}\| \varphi(UAU^{-1}) \| < \infty. \]
\end{question}

That the above questions do indeed make sense given the situation is justified in Proposition \ref{prop:S-and-S_V-as-pathwise-cap-or-cup}. We would like to highlight that the primary case of interest is the
one for which the map \(\varphi\) is given by a Hadamard product, and in particular the Hadamard product with a matrix \(J\)
whose entries are all one aside from the diagonal, where it is zero. In other words, while we state the question in this generality, the intended application is to the concrete situation where one
deletes the diagonal as mentioned in Question \ref{question:pathwise-question}.

Now, the above questions are still fairly tricky to answer at first glance---indeed, naively it is \emph{a priori} very hard to check how a pair \((A,Z)\) fares in any of them---so the strategy we take throughout this paper will be to study
the pairs \((A,Z)\) \emph{in families} as opposed to individually, and in particular, we will parametrize these families by the \(Z\)'s. In an effort for maximum generality, we will now state some definitions
which will help us organize the above in far more flexible circumstances than we are actually concretely interested in.

\begin{definition}\label{definition:frakN}
	Let \(V\) be a Banach algebra, let \(X\) be a Banach space, and let \(\varphi\!:V\to X\) be a bounded linear map. Define the functions
	\[ \frakN^\varphi_\cap\!:\End(V)\times\End(V)\to\R\cup\{\infty\}\quad\text{and}\quad\frakN^\varphi_\cup\!:\End(V)\times\End(V)\to\R\cup\{\infty\} \]
	by
	\begin{align*}
		\frakN^\varphi_\cap(a,z) &:=  \limsup_{r\to 0}\sup_{\substack{u\in V^\times, \\ \|u-z\|<r}}\|\varphi(uau^{-1})\| \\ 
		\frakN^\varphi_\cup(a,z) &:=  \limsup_{r\to 0}\inf_{\substack{u\in V^\times, \\ \|u-z\|<r}}\|\varphi(uau^{-1})\|.
	\end{align*}
	When \(X=V\) and \(\varphi=\id_V\), we write \(\frakN_\cap := \frakN_\cap^{\id_V}\) and \(\frakN_\cup := \frakN_\cup^{\id_V}\).
\end{definition}

So, choosing \(V=\End(\C^n)\), Question \ref{question:the-question} asks for which \(A\) and \(Z\) we have \(\frakN_\cap(A,Z)<\infty\), and similarly Question \ref{question:local-infimum-question} asks
for which \(A\) and \(Z\) we have \(\frakN_\cup(A,Z)<\infty\). More generally, the finiteness of the functions \(\frakN_\cap^\varphi\) and \(\frakN_\cup^\varphi\) allow us to express
Questions \ref{question:hadamard-question} and \ref{question:local-hadamard-infimum-question}.

\begin{definition}\label{definition:frakS}
	Let \(V\) be a Banach algebra, let \(X\) be a Banach space, and let \(\varphi\!:V\to X\) be a bounded linear map. For each \(z\in V\), define the sets
	\[ \frakS^\varphi_\cap(z) = \{ a\in V \mid \frakN^\varphi_\cap(a,z) < \infty \}\quad\text{and}\quad \frakS_\cup^\varphi(z) = \{ a\in V \mid \frakN_\cup^\varphi(a,z) < \infty \}. \]
	As before, we write \(\frakS_\cap(-) := \frakS^{\id_V}(-)\) and \(\frakS_\cup(-) := \frakS_\cup^{\id_V}(-)\).
\end{definition}
\begin{terminology}
	Elements of \(\frakS_\cap(z)\) and \(\frakS_\cup(z)\) are said to be \emph{sup-/inf-bounded at} \(z\), respectively (or just \emph{bounded at} \(z\) if context makes it clear which one is meant). Convsersely,
	elements in the complement are said to be \emph{sup-/inf-unbounded at} \(z\) (or just \emph{unbounded at} \(z\), subject to contextual clues).
\end{terminology}

Choosing \(V=\End(\C^n)\), we note that by introducing these sets, one ``reduces'' finding pairs \((A,Z)\) such that e.g. \(\frakN_\cup(A,Z)<\infty\) to computing the family \(\frakS_\cup(Z)\).
In this way, we can ``remove'' an entire dependency on \(A\) by blurring our eyes until they become simply the set \(\frakS_\cup(Z)\), which can be a useful tool for finding exploitable structure.
In essence, one hopes to turn answering the questions provided earlier into determining the ``algebra/calculus'' of the families \(\frakS_\cup(-)\).

Before getting into the weeds, there are some small comments to make.
\begin{itemize}[label=\(\star\)]
\item While we made a choice to work over \(\C\), everything in this paper works without alteration over \(\R\) as well.
\item A proof and result very similar to Lemma \ref{lemma:J-is-equivalent-to-id} appears in \cite{troedsson2024jointeigendecompositionmatrices}, though they were both developed independently based on the proof
in Appendix \ref{section:appendix-variant-of-modified}.
\end{itemize}

\subsection{Structure of the Paper \& Main Results}
We will begin with an outline of the overall structure of this paper before we disclose the main results. Because Question \ref{question:the-question} and Question \ref{question:local-infimum-question} require very different approaches,
we divide this paper up into sections to reflect this fact. However, before any of that, we give some preliminary notions in Section \ref{subsection:preliminaries} (most importantly the ``kernel criterion'').

In Section \ref{section:existence-of-bounded-paths}, we address Question \ref{question:local-infimum-question} and its generalization Question \ref{question:local-hadamard-infimum-question}.
In particular, the former is handled very elegantly in Section \ref{subsection:bounded-unmodified-existence}, where we give conditions under which the question can be answered
essentially completely, though we defer actually verifying the conditions until Section \ref{subsection:good-paths-in-hilbert-spaces-gen-ind-zero}. Inbetween these two, we have Section \ref{subsection:duality-in-coefficients},
discussing a duality phenomenon which arises from the difference between \(\|UAU^{-1}\|\) and \(\|U^{-1}AU\|\), which also provides us with a crucial tool in understanding the structure
of \emph{good paths,} which are the main devices we used to solve Question \ref{question:local-infimum-question} in Section \ref{subsection:bounded-unmodified-existence}.

In Section \ref{subsection:pathwise-boundedness}, we take a detour to talk about questions of boundedness along particular specified paths. When the path is given by some analytic function,
one can produce a refinement of the kernel criterion, which ends up giving some further rigidity to good paths (see Corollary \ref{corollary:good-path-rigidity-in-kernels}). Having done that,
we proceed to discussing Question \ref{question:local-hadamard-infimum-question}, giving two different kinds of approaches: in Section \ref{subsection:preorder}, we detail a relation between
modifiers \(\varphi\preceq\psi\) which descends to give comparisons like \(\frakS^\psi_\cup(-)\subseteq\frakS^\varphi_\cup(-)\), and this allows us, in favourable circumstances, to reduce computing
the more complex \(\frakS^\varphi_\cup(-)\) to the ordinary \(\frakS_\cup(-)\); in Section \ref{subsection:bounded-modified-existence}, besides applying the aforementioned relation, we develop
an alternate but also \emph{complete} description of \(\frakS_\cup^\varphi(-)\) inspired by the results of Sections \ref{subsection:bounded-unmodified-existence} \& \ref{subsection:good-paths-in-hilbert-spaces-gen-ind-zero}.

We end Section \ref{section:existence-of-bounded-paths} with Section \ref{subsection:modifiers-equivalent-to-the-identity}, giving a curious but suggestive consequence of the description
in Section \ref{subsection:bounded-modified-existence} that allows us to completely characterize those \(\varphi\) for which \(\frakS_\cup^\varphi(-)\) and \(\frakS_\cup(-)\) agree.

Section \ref{section:existence-of-unbounded-paths} is dedicated to addressing Question \ref{question:the-question} and Question \ref{question:hadamard-question}. The central observation
for our approach to these is given in Section \ref{subsection:unbounded-generalities}: that any element of \(\frakS^\varphi_\cap(Z)\) must also be in \(\frakS^\varphi_\cap(Z')\) for all sufficiently
nearby \(Z'\). With some trickery, this allows one to deduce that, at least in finite-dimensions, \(\frakS^\varphi_\cap(Z)\) can take only one of two forms, which we do in Section \ref{subsection:unbounded-modified-existence}.
In order to distinguish exactly what happens in our primary cases of interest, we use the same tools developed in Section \ref{subsection:preorder}.

Finally, we have two appendices: Appendix \ref{section:appendix-variant-of-modified}, giving an alternative (and more elementary) perspective on the results of Section \ref{subsection:preorder},
and Appendix \ref{section:appendix-notation}, where we provide some of the notation used throughout the paper.

While some of the main results have already been alluded to above, we now present them in a bit more detail. In Section \ref{subsection:good-paths-in-hilbert-spaces-gen-ind-zero}, we define what
it means for a bounded linear operator \(Z\) on a Banach space to be of \emph{generalized index zero.} These are operators \(Z\) which have closed image and for which there is
some invertible operator \(U\) such that \(UZ\) is self-adjoint (see Corollary \ref{corollary:gen-ind-zero-iff-closed-image-and-almost-self-adjoint}). This turns out to be a very important notion, and as suggested by the name, is a generalization
of the condition of being index zero. In particular, any index zero operator is also generalized index zero, and hence any linear map \(\C^n\to\C^n\) automatically satisfies this condition.
Operators with this property turn out to exhibit good behaviours. For example, an arbitrary operator need not be the limit of a sequence of invertible operators, but any generalized index zero
operator is, and we give a construction for this. The below is a combination of Theorem \ref{thm:good-linear-paths-equals-ker-equals-vee}, Corollary \ref{corollary:blow-up-for-almost-all},
and Theorem \ref{thm:hilbert-space-gen-index-zero-C1=C'}.
\begin{theorem}
	Let \(X\) be a Hilbert space, and let \(Z\in\sfB(X)\) be an operator of generalized index zero. Then there is a sequence of invertible operators converging to \(Z\), and
	\[ \frakS_\cup(Z) = \{ A\in\sfB(X)\mid A\ker{Z}\subseteq\ker{Z} \}. \]
	In particular, if \(X\) is finite-dimensional and \(Z\) is singular and non-zero, then for any sequence of invertibles \(U_n\to Z\) and almost all \(A\), we will have
	\[ \|U_n A U_n^{-1}\| \to \infty. \]
\end{theorem}

The proof of the above result leverages the existence of \emph{good paths} \(Z+tE(t)\). These are paths which are invertible near zero and of the form
\[ Z+tE(t) = Z + \sum_{k=1}^\infty t^kE_k\quad\text{where}\quad (Z+tE(t))^{-1} = C_{-1}t^{-1} + \sum_{k=0}^\infty t^kC_k. \]
In other words, they are invertible paths to \(Z\) with a prescribed form for their inverse. It is an independently interesting question to study the structure such objects. The below is a rephrasing of
Theorem \ref{thm:hilbert-space-gen-index-zero-C1=C'} and Corollary \ref{corollary:good-path-rigidity-in-kernels}, which also expands upon a statement in the last result.
\begin{theorem}
	Let \(X\) be a Hilbert space, \(Z\in\sfB(X)\) of generalized index zero. Then there exists a good path converging to \(Z\). Furthermore, an operator \(C\in\sfB(X)\)
	arises as the \((-1)\)th coefficient in the inverse of a good path converging to \(Z\) if and only if
	\[ \ker(C) = \img(Z),\quad \text{and}\quad \img(C) = \ker(Z). \]
	Moreover, if we have a good path \(Z + tE(t)\) as above, then there exists some $n > 0$ such that
	\[ \forall m < n,\, \ker{Z}\subseteq\ker{E_m},\quad\text{and}\quad \ker{Z}\cap\ker{E_n} = 0. \]
\end{theorem}

The above two results came out of our investigation of the ordinary, unmodified Question \ref{question:local-infimum-question}. For the more complex Question \ref{question:local-hadamard-infimum-question}, we have the following
result, which is a rephrasing of Theorem \ref{thm:hilbert-hadamard-general}.
\begin{theorem}
	Let \(X\) be a Hilbert space, \(Z\in\sfB(X)\) of generalized index zero, and let \(\varphi\!:\sfB(X)\to Y\) be a bounded linear map with codomain a Banach space \(Y\). Then \(A\in\frakS_\cup^\varphi(Z)\)
	if and only if there is some \(C\in\sfB(X)\) such that \(\img(C) = \ker(Z)\), \(\ker(C) = \img(Z)\), and \(\varphi(ZAC) = 0\).
\end{theorem}

This theorem suggests that it would be nice if one were able to simplify the collection of \(C\)'s one has to consider nicely (so that one could realistically compute \(\frakS^\varphi_\cup(Z)\)), but to the authors' knowledge there is
no good way to do this. On the other hand, since the ordinary \(\frakS_\cup(Z)\) is much more computable, one may wonder under what conditions \(\frakS_\cup^\varphi(Z) = \frakS_\cup(Z)\) for all \(Z\) of generalized index zero. While
Section \ref{subsection:preorder} gives one approach to this kind of question, the theorem above actually yields a property completely characterizing these \(\varphi\)'s. The below is
a combination of Theorem \ref{thm:S_V^phi=S_V-implies-phi(2nd-ord-nilp)-nonzero} and Corollary \ref{corollary:hadamard-phis-equiv-to-id}.
\begin{theorem}
	Let \(X\) be a Hilbert space, \(Y\) a Banach space, and let \(\varphi\!:\sfB(X)\to Y\) be a bounded linear map. Then \(\frakS_\cup^\varphi(Z) = \frakS_\varphi(Z)\) for
	all \(Z\in\sfB(X)\) of generalized index zero if and only if for all non-zero \(T\in\sfB(X)\) such that \(T^2=0\), we have \(\varphi(T)\not=0\).

	In particular, if \(X\) is finite-dimensional and \(\varphi\) is given by taking the Hadamard product \(A\mapsto H*A\) with some matrix \(H\), then \(\frakS_\cup^\varphi(Z) = \frakS_\cup(Z)\)
	for all \(Z\) if and only if all non-diagonal entries in \(H\) are non-zero.
\end{theorem}

Here, we will interject and mention one of the ingredients in proving the results about \(\frakS_\cup(-)\): a slightly sharpened polar decomposition for operators of generalized index zero.
In general, polar decomposition provides for you a decomposition of an operator into a partial isometry and a self-adjoint operator. However, when the operator is assumed to be of generalized index zero,
this can be strengthened to where the partial isometry is actually invertible. Below is a slight paraphrasing of Theorem \ref{thm:polar-decomposition}.
\begin{theorem}[Sharpened polar decomposition]
	Let \(X\) be a Hilbert space, and let \(Z\in\sfB(X)\) be of generalized index zero.
	Then there is some \(U\in\sfB(X)^\times\), which is also a partial isometry, such that \( Z = U\sqrt{Z^* Z} \). Furthermore, under some conditions, \(U\) is unitary.
\end{theorem}

For the purposes of Section \ref{section:existence-of-unbounded-paths}, we remark that Proposition \ref{prop:S(invertible)-is-everything} below implies that we may always restrict our attention
solely to singular operators. With this, we come to the of the main theorems of Section \ref{section:existence-of-unbounded-paths}, one of which unfortunately only works in finite-dimensions. The statement
below is a combination of Theorem \ref{thm:S^phi(Z)-classification} and Theorem \ref{thm:S(singular)-is-cI}.
\begin{theorem}
	Let \(X\) be a Hilbert space, and let \(Z\in\sfB(X)\) be singular. Then \(\frakS_\cap(Z) = \{\lambda I\mid \lambda\in\C\}\).

	If \(X\cong\C^n\) is finite-dimensional and \(\varphi\!:\End(\C^n)\to\End(\C^n)\) is some bounded linear map, then either \(\frakS^\varphi_\cap(Z) = \{\lambda I \mid \lambda\in\C\}\) or \(\frakS^\varphi_\cap(Z) = \End(\C^n)\).
\end{theorem}

The theorem provides us with two possibilities for \(\frakS^\varphi(Z)\), but we have no immediate way to tell which one will hold. However, we would still like to know what \(\frakS^\varphi(Z)\) looks like for at least
\emph{some} more interesting choices of \(Z\), for which we have the following, which is a paraphrasing of Theorem \ref{thm:hadamard-sup-version-J-class}.
\begin{theorem}
	Let \(\varphi\!:\End(\C^n)\to\End(\C^n)\) be given by \(A\mapsto J*A\), where all non-diagonal elements of \(J\) are non-zero, and let \(Z\in\End(\C^n)\). Then \(\frakS^\varphi_\cap(Z) = \frakS_\cap(Z)\).
	In particular, suppose \(Z\) is and singular. Then \(\frakS^\varphi_\cap(Z) = \{\lambda I \mid \lambda\in\C\}\).
\end{theorem}

\subsection{Some Basic Preliminaries}\label{subsection:preliminaries}
We collect some miscellaneous preliminary notions which will be useful throughout the paper. To start, we want to justify that the constructions \(\frakS^\varphi_\cap(-)\) and \(\frakS_\cup^\varphi(-)\)
really do what we claim. To this end, we introduce a much more fine-grained set specifically containing information about a particular path.

\begin{definition}\label{definition:frakS_pw}
	Let \(V\) be a Banach algebra, \(X\) a Banach space, \(\varphi\!:V\to X\) a bounded linear map, let \(z\in V\), and consider a sequence \(u_k \to z\), with \(u_k\in V^\times\) for all \(k\). Define the vector space
	\[ \frakS_\pw^\varphi(z; u_k) := \left\{ a\in V \bigmid \lim_{k\to\infty} \|\varphi(u_kau_k^{-1})\| < \infty \right\}. \]
	If the path \(u_k\to z\) is determined by some function \(z + e(t)\) where \(e(t)\to 0\) as \(t\to 0\), we will alternatively write \(\frakS_\pw^\varphi(z; z+e(t))\) or \(\frakS_\pw^\varphi(z; e(t))\).
	If \(X=V\) and \(\varphi=\id\), we write \(\frakS_{\pw}(-)\) instead of \(\frakS_{\pw}^{\id}(-)\).
\end{definition}
\begin{remark}
	The set \(\frakS_\pw^{\varphi}(z; u_k)\) of Definition \ref{definition:frakS_pw} captures boundedness along the given path \(u_k \to z\), and we have
	\[ \frakS^{\varphi}_\cap(z) \subseteq \bigcap_{u_k \to z}\frakS_\pw^{\varphi}(z;u_k) = \bigcap_{z+e(t)}\frakS_\pw^\varphi(z;e(t)). \]
\end{remark}

\begin{proposition}\label{prop:S-and-S_V-as-pathwise-cap-or-cup}
	Let \(V\) be a Banach algebra, \(X\) a Banach space, \(\varphi\!:V\to X\) and let \(z\in V\). Then
	\[ \frakS^\varphi_\cap(z) = \bigcap_{u_k\to z}\frakS_{\pw}^\varphi(z;u_k)\quad\text{and}\quad \frakS_\cup^\varphi(z) = \bigcup_{u_k\to z}\frakS_{\pw}^\varphi(z;u_k). \]
\end{proposition}
\begin{proof}
For the left assertion, one inclusion is clear, as remarked above. For the other inclusion, we proceed by proving the contrapositive: suppose
\(\frakN^\varphi_\cap(a,z)=\infty\). Then we know that
\[ \forall r > 0,\quad \sup_{\substack{\|u-z\| < r \\ u\in V^\times}}\|\varphi(uau^{-1})\| = \infty. \]
In particular, we may define a sequence of invertibles \(u_k \to z\) by choosing \(u_k\in V^\times\) such that \(\|u_k-z\| < 1/k\) and \(\|u_kau_k^{-1}\| > k\). Then, by construction
\[ \lim_{k\to\infty} \|u_kau_k^{-1}\| = \infty \]
so that \(a\not\in\frakS^\varphi_{\pw}(z;u_k)\).

To prove the right assertion, note that one easily sees that
\[ \bigcup_{u_k\to z}\frakS_{\pw}^\varphi(z;u_k) \subseteq \frakS_\cup^\varphi(z). \]
For the other inclusion, let \(A\in\frakS_\cup^\varphi(z)\), and let \(\lambda = \frakN_\cup^\varphi(a,z) < \infty\). This tells us, in particular, that
\[ \forall r>0,\quad \inf_{\substack{\|u-z\| < r \\ u\in V^\times}}\|uau^{-1}\| \leq \lambda, \]
so we may pick \(u_k\in V^\times\) such that \(\|u_k - z\| < 1/k\) and \(\|u_kau_k^{-1}\| \to \lambda\),
so that \(a\in\frakS_{\pw}^\varphi(z;u_k)\).
\end{proof}

This yields a number of corollaries. An immediate one is that we may restrict our attention entirely to singular operators, as demonstrated below.

\begin{proposition}\label{prop:S(invertible)-is-everything}
	Let \(V\) be a Banach algebra. If \(u\in V^\times\) is invertible, then
	\[ \frakS^\varphi_\cap(u) = V. \]
\end{proposition}
\begin{proof}
If \(u\) is invertible and \(u_k\to u\) is a sequence of invertibles, then \(u_k^{-1}\to u^{-1}\) and for any \(a\in V\) we have
\[ u_k a u_k^{-1} \to u a u^{-1}.\]
In particular,
\[ \lim_{k\to\infty}\|u_kau_k^{-1}\| = \|uau^{-1}\| < \infty \]
so by Proposition \ref{prop:S-and-S_V-as-pathwise-cap-or-cup} we are done.
\end{proof}

Here is a basic observation about the structure of the family \(\frakS_{\pw}^\varphi(-)\), and in particular \(\frakS_{\pw}(-)\), which together imply some structure
on \(\frakS^\varphi_\cap(-)\) following from Proposition \ref{prop:S-and-S_V-as-pathwise-cap-or-cup}.
\begin{theorem}\label{thm:local-algebra}
	Let \(V\) be a Banach algebra, let \(z\in V\), and let \(u_k\to z\) be a sequence of invertibles.
	\begin{enumerate}[label=(\alph*)]
	\item For a Banach space \(X\) and bounded linear map \(\varphi\!:V\to X\), the set \(\frakS_{\pw}^\varphi(z)\) forms a vector space over \(\C\).
	\item For a Banach algebra \(W\) and continuous algebra map \(\varphi\!:V\to W\), the vector space \(\frakS_{\pw}^\varphi(z)\) forms an algebra over \(\C\).
	\end{enumerate}
	In particular, \(\frakS^\varphi_\cap(z)\) is always vector space, and choosing \(V=W\) and \(\varphi=\id\), we see that \(\frakS_\cap(z)\) is an algebra.
\end{theorem}
\begin{proof}
Part (a) follows immediately from the linearity of \(\varphi\), and the subadditivity and homogeneity of the norm \(\|\cdot\|\). For part (b), since \(\varphi\) is
an algebra map, we have
\[ \|\varphi(u_kabu_k^{-1})\| = \| \varphi(u_kau_k^{-1}\cdot u_kbu_k^{-1}) \| \leq \| \varphi(u_kau_k^{-1}) \|\cdot \|\varphi(u_kbu_k^{-1})\| \]
from which the result follows.

The final statement follows immediately by applying Proposition \ref{prop:S-and-S_V-as-pathwise-cap-or-cup} since the intersection of vector spaces (resp.\ algebras) is a vector space (resp.\ an algebra).
\end{proof}

Another interesting property of the ``pathwise spaces'' is that one can relate boundedness along one path to boundedness along another. This fact will
implicitly be used in Section \ref{section:existence-of-unbounded-paths}.
\begin{theorem}\label{thm:pathwise-hadamard-conjugation-law}
	Let \(V\) be a Banach algebra, let \(\varphi\!:V\to X\) be a bounded linear map, let \(z\in V\), and let \(u_k\to z\) be a sequence of invertibles. Then, for all \(p\in V^\times\), we have
	\[ \frakS_{\pw}^\varphi(zp; u_kp) = p^{-1} \frakS_\cup^\varphi(z;u_k) p\quad\text{and}\quad\frakS_{\pw}(pz;pu_k) = \frakS(z;u_k). \]
\end{theorem}
\begin{proof}
Simply observe that we may write
\[ \lim_{k\to\infty} \| \varphi(u_kau_k^{-1}) \| = \lim_{k\to\infty} \| \varphi\!\left( (u_kp)(p^{-1} a p)(u_kp)^{-1}\right)\! \|. \]
Therefore, \(a\in\frakS_{\pw}^\varphi(z;u_k)\) if and only if \(p^{-1}ap\in\frakS_{\pw}^\varphi(zp;u_kp)\), and \(b\in\frakS_\cup^\varphi(zp;u_kp)\) if and only if \(pbp^{-1}\in\frakS_{\pw}^\varphi(z;u_k)\),
so that \(\frakS_{\pw}^\varphi(zp;u_kp) = p^{-1}\frakS_{\pw}^\varphi(z;u_k)p\). This proves the first statement.

For the second statement, note that for all \(x\in V\), we have
\[  \frac{1}{\| p \|\cdot\| p^{-1} \|} \cdot \| p x p^{-1} \| \leq \| x \| \leq \| p^{-1} \| \cdot \| p \| \cdot \| p x p^{-1}\|. \]
Therefore,
\[ \frac{1}{\| p \|\cdot\| p^{-1} \|} \cdot \| (p u_k) a (pu_k)^{-1} \| \leq \| u_k a u_k^{-1} \| \leq \| p^{-1} \| \cdot \| p \| \cdot  \| (p u_k) a (pu_k)^{-1}\| \]
which yields the result.
\end{proof}
\begin{corollary}\label{corollary:hadamard-conjugation-law}
	Let \(V\) be a Banach algebra, let \(\varphi\!:V\to X\) be a bounded linear map, and let \(z\in V\). Then, for all \(p\in V^\times\), we have
	\[ \frakS^\varphi_\cap(zp) = p^{-1} \frakS^\varphi(z) p\quad\text{and}\quad\frakS^\varphi_\cup(zp) = p^{-1} \frakS^\varphi_\cup(z) p. \]
\end{corollary}
\begin{proof}
Observe that there is a bijection between the sets
\[ \{ \text{sequences } u_k\to z \}\quad\text{and}\quad \{ \text{sequences } v_k\to zp \} \]
given by multiplication on the right by \(p\) or \(p^{-1}\). In particular, applying Proposition \ref{prop:S-and-S_V-as-pathwise-cap-or-cup} and Theorem \ref{thm:pathwise-hadamard-conjugation-law}, we have
\begin{align*}
	\frakS^\varphi(zp) = \bigcap_{v_k\to zp}\frakS_{\pw}^\varphi(zp;v_k) &= \bigcap_{u_k\to z}\frakS_{\pw}^\varphi(zp;u_kp)\\
	&= \bigcap_{u_k\to z}p^{-1}\frakS_{\pw}^\varphi(z;u_k)p = p^{-1}\left(\bigcap_{u_k\to z}\frakS_{\pw}^\varphi(z;u_k)\right)p = p^{-1} \frakS^\varphi(z) p.
\end{align*}
The proof for \(\frakS_\cup^\varphi(-)\) is identical.
\end{proof}

Clearly, if \(\varphi\) is some arbitrary bounded linear map, the problem of computing the families \(\frakS^\varphi_\cap(-)\) and \(\frakS_\cup^\varphi(-)\) is harder than computing
\(\frakS_\cap(-)\) and \(\frakS_\cup(-)\), so these are what we begin with. A starting point is a sufficient condition for being unbounded:

\begin{proposition}\label{prop:local-erik}
	Let \(X\) be a Banach space, and let \(A,Z\in\sfB(X)\). If the two equivalent conditions
	\begin{enumerate}[label=(\roman*)]
	\item \(\ker(Z)\not\subseteq\ker(ZA)\),
	\item \(A\ker(Z)\not\subseteq\ker(Z)\),
	\end{enumerate}
	are satisfied, then \(\frakN_\cup(A,Z)=\frakN_\cap(A,Z)=\infty\). In particular, \(A\not\in\frakS_\cup(Z)\) and \(A\not\in\frakS_\cap(Z)\).
\end{proposition}
\begin{proof}
That the two conditions are equivalent is simply the following observation:
\begin{align*}
	\ker(Z)\not\subseteq \ker(ZA) &\iff \exists x\in\ker(Z)\text{ s.t. }ZAx\not=0 \\
	&\iff \exists x\in\ker(Z)\text{ s.t. } Ax\not\in\ker(Z) \\
	&\iff A(\ker(Z)) \not\subseteq \ker(Z)
\end{align*}
Now, suppose we have some \(x\not=0\) in \(X\) such that \(Zx = 0\) but \(ZAx \not= 0\). Then
\begin{align*}
	\limsup_{r\to 0}\inf_{\substack{U\in\sfB(X)^\times, \\ \|U-Z\|<r}}\left(\sup_{w\not=0}\frac{\|UAw\|}{\|Uw\|}\right) &\geq
	\limsup_{r\to 0}\inf_{\substack{U\in\sfB(X)^\times, \\ \|U-Z\|<r}}\frac{\|UAx\|}{\|Ux\|} \\
	&= \limsup_{r\to 0}\inf_{\substack{U\in\sfB(X)^\times, \\ \|U-Z\|<r}}\frac{\|UAx\|}{\|(U-Z)x\|} \\
	&\geq \limsup_{r\to 0}\inf_{\substack{U\in\sfB(X)^\times, \\ \|U-Z\|<r}}\frac{1}{r}\frac{\|UAx\|}{\|x\|} = \infty \\
\end{align*}
where the equality is because \(Ux = Ux - 0 = Ux - Zx\), and where we note that \(\|UAx\|\not\to 0\) as \(r\to 0\) since \(ZAx \not= 0\).
We conclude that \(\frakN_\cup(A,Z)=\infty\). Finally, it is clear that by the definitions of \(\inf\) and \(\sup\) that \(\frakN_\cup(A,Z)\leq\frakN_\cap(A,Z)\), so that the latter is also infinite.
\end{proof}

From the above, we introduce the following gadget to allow us to fit the criterion provided into the general formalism we are working with.

\begin{definition}\label{def:kernel-invariant-set}
	Let \(X\) be a Banach space. For any \(Z\in\sfB(X)\), define the set
	\[ \frakS_{\ker}(Z) := \{ A\in\sfB(X)\mid A\ker(Z)\subseteq\ker(Z) \}, \]
	i.e.\ the set of operators which keep the kernel of \(Z\) invariant.
\end{definition}
\begin{theorem}\label{thm:kernel-invariant-algebra}
	Let \(X\) be a Banach space, and let \(Z\in\sfB(X)\). Then
	\begin{enumerate}[label=(\alph*)]
	\item \(\frakS_{\ker}(Z)\) forms an algebra, and
	\item we have the inclusions
	\[ \{\lambda I \mid \lambda\in\C\} \subseteq  \frakS_\cap(Z)\subseteq \frakS_\cup(Z)\subseteq \frakS_{\ker}(Z). \]
	\end{enumerate}
\end{theorem}
\begin{proof}
(a) Let \(A,B\in\frakS_{\ker}(Z)\). Then, clearly,
\[ AB\ker(Z)\subseteq A\ker(Z)\subseteq \ker(Z) \]
so that \(AB\in\frakS_{\ker}(Z)\). Similarly, if \(\lambda,\mu\in\C\) and \(x\in\ker(Z)\), we have \(Z(\lambda A+\mu B)x = \lambda ZAx + \mu ZBx = 0\),
so \((\lambda A+\mu B)x\in\ker(Z)\), and hence \(\lambda A +\mu B \in \frakS_{\ker}(Z)\).

(b) The first inclusion is trivial since $\lambda I$ commutes with everything. That \(\frakS_\cap(Z)\subseteq\frakS_\cup(Z)\) is immediate since \(\frakN_\cup(A,Z)\leq\frakN_\cap(A,Z)\), and that \(\frakS_\cup(Z)\subseteq\frakS_{\ker}(Z)\)
is a consequence of Proposition \ref{prop:local-erik}.
\end{proof}

\clearpage
\section{Existence of Bounded Paths}\label{section:existence-of-bounded-paths}
The goal of this section is to treat the computation or characterization of \(\frakS_\cup(-)\) and \(\frakS_\cup^\varphi(-)\) in as great a generality as we can muster. Since the latter
is harder, more finicky, and generally less elegant, we will start with the former.

\subsection{Existence Without a Modifier}\label{subsection:bounded-unmodified-existence}
Our general approach is to make use of the algebra \(\frakS_{\ker}(Z)\) described in Definition \ref{def:kernel-invariant-set}.
In principle, since Theorem \ref{thm:kernel-invariant-algebra} tells us that \(\frakS_{\cup}(Z)\subseteq\frakS_{\ker}(Z)\), we could simply compute e.g.\ the dimension of \(\frakS_{\ker}(Z)\) and be done (at least in
the finite-dimensional setting). However, we will instead prove the stronger result that \(\frakS_\cup(Z)\) and \(\frakS_{\ker}(Z)\) are \emph{equal,} in essentially full generality, and only then compute \(\frakS_{\ker}(Z)\). To do this,
we will pass through yet another intermediary. The idea here is as follows: the algebra \(\frakS_\cup(Z)\) concerns boundedness properties of arbitrary paths converging to \(Z\); we will consider the more modest
setting of a path converging to \(Z\) which has a prescribed form for its inverse, allowing us to deduce boundedness along that particular path very easily.

\begin{terminology}\label{terminology:good-path}
	Let \(V\) be a Banach algebra, and let \(z\in V\).
	\begin{enumerate}[label=(\arabic*)]
	\item An \emph{admissible path} to \(z\) is a path \(\R_{>0}\to V\), invertible near zero, of the form
	\[ z+t\cdot e(t), \]
	where \(e(t)\) is analytic for sufficiently small \(t\), i.e.\ is of the form
	\[ e(t) = \sum_{k=0}^\infty e_k t^k.  \]
	\item We say an admissible path \(z+te(t)\) is \emph{linear} if \(e(t)\) is a constant function.
	\item An admissible path \(z+te(t)\) is \emph{good} if the inverse is of the form
	\[ c_{-1}t^{-1} + c(t) \]
	where \(c(t)\) is analytic.
	\end{enumerate}
\end{terminology}
\begin{notation}
	If \(C(t)\) is some analytic function, we will write \(C_k\) for the \(k\)th coefficient in the series expansion of \(C(t)\).
\end{notation}

Of course, one can write down any definition one wants, and it could all be for nothing if the set of things satisfying the definition is empty. However, we defer the issue of
good paths for Section \ref{subsection:good-paths-in-hilbert-spaces-gen-ind-zero}, focusing instead on how they ought to be used. We make a definition, the succeeding Proposition \ref{prop:general-good-paths-kill-Z} being the motivation:

\begin{definition}\label{definition:frakS_c(z)}
	Let \(V\) be a Banach algebra, let \(z\in V\), and let \(c\in V\) be such that \(zc = cz = 0\). Define the algebra
	\[ \frakS_{c}(z) := \{ a\in V\mid zac = 0 \}. \]
\end{definition}

\begin{lemma}\label{lemma:Sker-in-SC}
	Let \(X\) be a Banach space, and let \(Z,C\in\sfB(X)\) be such that \(ZC = CZ = 0\). Then
	\[\frakS_{\ker}(Z)\subseteq\frakS_{C}(Z).\]
\end{lemma}
\begin{proof}
Saying that \(ZC = 0\) is the same as saying that \(\img{C}\subseteq\ker{Z}\). Now, if \(A\in\frakS_{\ker}(Z)\), then \(A\) keeps the kernel
of \(Z\) invariant, and hence
\[ A\img(C) \subseteq A\ker(Z)\subseteq\ker(Z)\implies ZAC=0 \]
as desired.
\end{proof}

\begin{proposition}\label{prop:general-good-paths-kill-Z}
	Let \(V\) be Banach algebra, let \(z\in V\), and let \(z+te(t) \in \calE(z)\) be a good path to \(z\) with inverse \(c_{-1}t^{-1} + c(t)\). Then \(c_{-1}z=zc_{-1}=0\).
\end{proposition}
\begin{proof}
We have that
\[ (z+te(t))(c_{-1}t^{-1}+c(t)) = I = (c_{-1}t^{-1}+c(t))(z+te(t)). \]
Expanding \(e(t)\) and \(c(t)\) in their series forms, then multiplying this out and comparing coefficients, yields the result.
\end{proof}

These good paths have predictable boundedness properties, characterized by the simple algebraic criterion of \(\frakS_{C_{-1}}(Z)\),
and furthermore, Lemma \ref{lemma:Sker-in-SC} relates this to \(\frakS_{\ker}(Z)\). Using this, we obtain the following lemma, and one of the main
theorems of this section.

\begin{lemma}\label{lemma:good-path-annihilates-coefficient-implies-vee}
	Let \(V\) be a Banach algebra, let \(z\in V\), and let \(z+te(t)\) be a good path with inverse \(c_{-1}t^{-1}+c(t)\). Then \(\frakS_{c_{-1}}(z)\subseteq \frakS_{\pw}(z;z+te(t))\). In particular,
	\(\frakS_{c_{-1}}(z)\subseteq  \frakS_\cup(z)\).
\end{lemma}
\begin{proof}
Suppose \(a\in\frakS_{c_{-1}}(z)\). Note that
\[ \frakN_\cup(a,z)=\limsup_{r\to 0}\inf_{\substack{u\in V^\times \\ \|u-z\|<r}}\| uau^{-1} \| \leq \lim_{t\to 0}\|(z+te(t))a(z+te(t))^{-1}\|. \]
Writing out the latter using the fact that we have an explicit inverse and that \(zac_{-1}=0\), we get
\begin{align*}
	(z+te(t))a(c_{-1}t^{-1} + c(t)) &= zac_{-1}t^{-1} + zac(t) + e(t)ac_{-1} + e(t)ac(t)\cdot t \\
	&= zac(t) + e(t)ac_{-1} + e(t)ac(t)\cdot t = O(1)
\end{align*}
which is bounded in norm as \(t\to 0\). Finally, see Proposition \ref{prop:S-and-S_V-as-pathwise-cap-or-cup}.
\end{proof}

\begin{theorem}\label{thm:good-linear-paths-equals-ker-equals-vee}
	Let \(X\) be a Banach space, let \(Z\in \sfB(X)\), and suppose that there is some good path \(Z+tE(t)\) going to \(Z\). Then
	\[ \frakS_{\ker}(Z) = \frakS_{C_{-1}}(Z) = \frakS_{\pw}(Z;Z+tE(t)) = \frakS_\cup(Z), \]
	where \((Z+tE(t))^{-1} = C_{-1}t^{-1}+C(t)\). In particular, \(\frakS_{C_{-1}}(Z)\) is independent of \(E(t)\).
\end{theorem}
\begin{proof}
By Lemma \ref{lemma:Sker-in-SC} and Lemma \ref{lemma:good-path-annihilates-coefficient-implies-vee} together with
part (b) of Theorem \ref{thm:kernel-invariant-algebra}, we have
\[ \frakS_{\ker}(Z)\subseteq \frakS_{C_{-1}}(Z) \subseteq \frakS_{\pw}(Z;Z+tE(t)) \subseteq \frakS_{\cup}(Z) \subseteq \frakS_{\ker}(Z) \]
which yields the result.
\end{proof}
\begin{corollary}\label{corollary:S-vee-equals-S-ker}
	Let \(Z\in\End(\C^n)\). Then \(\frakS_\cup(Z)=\frakS_{\ker}(Z)\).
\end{corollary}
\begin{proof}
By Theorem \ref{thm:good-linear-paths-equals-ker-equals-vee}, if a good path exists for \(Z\), then the result holds. By our later Theorem \ref{thm:hilbert-space-gen-index-zero-C1=C'}
(see also Theorem \ref{thm:properties-of-C1(Z)} and Corollary \ref{corollary:hilbert-space-S_V=S_ker}), good paths exist for all \(Z\).
\end{proof}

\begin{remark}
	Before moving on, we would like to highlight the following remarkable observation: Corollary \ref{corollary:S-vee-equals-S-ker} implies
	that for all \(Z\in\End(\C^n)\), \(\frakS_\cup(Z)\), which was \emph{a priori only a set,} is actually an \emph{algebra.} We do not know any other way to show this,
	aside from as a corollary as just explained.
\end{remark}

So the good paths have allowed us to show that \(\frakS_\cup(Z)=\frakS_{\ker}(Z)\) (momentarily assuming that they exist; when we establish that they do, we will also be given a more
general veresion of Corollary \ref{corollary:S-vee-equals-S-ker}). However, for this to be a useful computation in practice,
we need to know what \(\frakS_{\ker}(Z)\) is, or at least the dimension. To solve this issue, we will reduce the problem to computing \(\frakS_{\ker}(Z)\)
for just a few choices of \(Z\) where one can do a calculation by hand.

\begin{theorem}\label{thm:S-vee-ker-operations}
	Let \(X\) be a Banach space, let \(Z\in\sfB(X)\), and let \(M\in\sfB(X)^\times\). Then
	\begin{enumerate}[label=(\alph*)]
	\item \(M^{-1}\frakS_{\ker}(Z)M = \frakS_{\ker}(ZM)\),
	\item \(\frakS_{\ker}(MZ) = \frakS_{\ker}(Z)\), and
	\item for any other \(M'\in\sfB(X)^\times\), there is an isomorphism \(\frakS_{\ker}(Z)\cong\frakS_{\ker}(MZM')\).
	\end{enumerate}
	In particular, if \(X=\C^n\), the dimension of \(\frakS_{\ker}(Z)\) depends only on the rank of \(Z\).
\end{theorem}
\begin{proof}
Observe that \(\ker(MZ) = \ker(Z)\), so (b) is trivial. To see that (a) holds, note first that \(x\in\ker(ZM)\) if and only if \(Mx\in\ker(Z)\), so that
\(\ker(ZM) = M^{-1}\ker(Z)\). Now, if \(A\in\frakS_{\ker}(Z)\), then \(M^{-1}AM\) satisfies
\[ M^{-1}AM\ker(ZM) = M^{-1}A\ker(Z)\subseteq M^{-1}\ker(Z) = \ker(ZM). \]
Therefore, \(M^{-1}\frakS_{\ker}(Z)M\subseteq\frakS_{\ker}(ZM)\). Conversely, if \(B\in\frakS_{\ker}(ZM)\), then
\[ MBM^{-1}\ker(Z) = MB\ker(ZM) \subseteq M\ker(ZM) = \ker(Z) \]
so that \(M\frakS_{\ker}(ZM)M^{-1}\subseteq\frakS_{\ker}(Z)\), that is \(\frakS_{\ker}(ZM)\subseteq M^{-1}\frakS_{\ker}(Z)M\).

Part (c) is a formal consequence of (a) and (b). In particular, we have
\[\frakS_{\ker}(Z) = \frakS_{\ker}(MZ) \cong \frakS_{\ker}(MZM'), \]
where the equality is by (b) and the isomorphism is by (a).
\end{proof}

\begin{remark}
	A priori, computing \(\frakS_{\ker}(Z)\) is very hard: it requires identifying the invariant subspaces of a matrix. However, since \(\frakS_{\ker}(Z)\) only cares
	about the rank of \(Z\), we may now replace \(Z\) entirely by the diagonal matrix with \(m=\dim\img{Z}\) one's on the diagonal and everything else zero. Due to
	the simplicity of this matrix, it is now much more tractable to compute \(\frakS_{\ker}(Z)\), at least up to an isomorphism given by conjugation.
\end{remark}

\begin{notation}
	Let \(m\leq n\). We write \(D_{n,m}\) for the \(n\times n\) diagonal matrix with \(m\) ones on the diagonal and zeros elsewhere. That is,
	\[ D_{n,m} := \begin{pmatrix} I_{m} & 0_{m\times (n-m)} \\ 0_{(n-m)\times m} & 0_{(n-m)\times (n-m)} \end{pmatrix} = I_m \oplus 0_{(n-m)\times (n-m)} . \]
\end{notation}

\begin{proposition}\label{prop:Sker(Dnm)-calculation}
	Consider \(\C^n\), and let \(n = m + k\). We have
	\[ \frakS_{\ker}(D_{n,m}) = \bigg\{ \begin{pmatrix} X & 0_{m\times k} \\ Y & W \end{pmatrix} \bigmid X \in \C^{m\times m},\,
	Y\in \C^{k\times m},\text{ and } W\in\C^{k\times k} \bigg\}. \]
	In particular, for any singular \(Z\) with \(\dim\img{Z} = m\), we have \(\dim\frakS_{\ker}(Z) = n^2 - mn + m^2\).
\end{proposition}
\begin{proof}
It is easily seen that \(\ker(D_{n,m})\) is given by all vectors of the form \(v = 0_{m}\oplus v'\), where \(v'\) is of dimension \(k\). Thus, writing
\[ A = \begin{pmatrix} X & W' \\ Y & W \end{pmatrix},\quad W'\in\C^{m\times k}, \]
to have \(Av\in\ker(Z)\), we must have that
\[ Av = \begin{pmatrix} X & W' \\ Y & W \end{pmatrix} \begin{pmatrix} 0_m \\ v' \end{pmatrix} = \begin{pmatrix} W'v' \\ Wv' \end{pmatrix}\in\ker(Z), \]
which then means we must have \(W'v' = 0\). However, \(v'\) was totally arbitrary, so this implies that \(W'=0\).

The final assertion is just the observation that
\begin{align*}
	\dim\frakS_{\ker}(D_{n,m}) &= m^2 + mk + k^2 \\
	&= m^2 + m(n-m) + (n-m)^2 = n^2 - mn + m^2
\end{align*}
and that \(\dim\frakS_{\ker}(Z)\) only depends on the rank of \(Z\).
\end{proof}

\begin{corollary}\label{corollary:blow-up-for-almost-all}
	Let \(Z\in\End(\C^n)\), and let \(m=\dim\img{Z}\). Then \(\frakS_\cup(Z)\) is a linear subspace of \(\End(\C^n)\) of dimension
	\(n^2 - mn + m^2\). In particular, if \(Z\) is singular, then for almost all \(A\in\End(\C^n)\) we have \(\frakN_\cup(A,Z)=\infty\).
\end{corollary}
\begin{proof}
Apply Corollary \ref{corollary:S-vee-equals-S-ker} and Proposition \ref{prop:Sker(Dnm)-calculation} to see that
\[ \dim\frakS_\cup(Z) = \dim\frakS_{\ker}(Z) = n^2 - mn + m^2 \]
as desired.
\end{proof}

\subsection{Duality in Coefficients}\label{subsection:duality-in-coefficients}
We begin this subsection with the following remarkable observation:

\begin{proposition}\label{prop:general-good-path-duality}
	Let \(V\) be a Banach algebra, let \(z\in V\), and let \(z+te(t)\) be a good path to \(z\) with inverse \(c_{-1}t^{-1}+c(t)\). Then \(c_{-1} + tc(t)\) is a good path with
	inverse \(zt^{-1}+e(t)\).
\end{proposition}
\begin{proof}
Note that, for small enough \(t\), we have
\[ 1 = (z+te(t))\cdot (c_{-1}t^{-1}+c(t)) = (zt^{-1} + e(t))t\cdot (c_{-1}t^{-1} + c(t)) = (zt^{-1}+e(t))\cdot(c_{-1}+tc(t)) \]
which yields the result.
\end{proof}

That is, one sees that there is some kind of duality between good paths going to some \(z\) and good paths going to the \((-1)\)th coefficient of a good path going to \(z\).
That this is the case suggests that there is some kind of deeper relationship at play, and the purpose of this section is to uncover at least some aspects of this.

\begin{definition}\label{definition:frakC}
	Let \(V\) be a Banach algebra, and let \(z\in V\). Define the set \(\frakC_{-1}(z)\) to be the set of all \(c_{-1}\in V\) for which there exists a good path \(z+te(t)\)
	going to \(z\) such that \((z+te(t))^{-1} = c_{-1}t^{-1} + c(t) \).
\end{definition}


\begin{theorem}\label{thm:properties-of-C1(Z)}
	Let \(V\) be a Banach algebra.
	\begin{enumerate}[label=(\alph*)]
	\item Let \(c,z\in V\). Then \(c\in\frakC_{-1}(z)\) if and only if \(z\in\frakC_{-1}(c)\).
	\item Let \(z\in V\). Then \(\frakC_{-1}(z)\) is non-empty if and only if there exists a good path converging to \(z\).
	\item Let \(z\in V\). For any \(c\in\frakC_{-1}(z)\), we have \(cz = zc = 0\).
	\item Let \(z\in V\), and let \(p\in V^\times\). Then \(\frakC_{-1}(zp) = p^{-1}\frakC_{-1}(z)\) and \(\frakC_{-1}(pz) = \frakC_{-1}(z)p^{-1}\).
	\end{enumerate}
\end{theorem}
\begin{proof}
(a) is a reformulation of Proposition \ref{prop:general-good-path-duality}, and (c) is a reformulation of Proposition \ref{prop:general-good-paths-kill-Z}.
Part (b) is just by definition. To obtain (d), note that, given a good path \(z+te(t)\), multiplication
by \(p\) on the left and right yields good paths \(zp + te(t)p\) and \(pz + t p e(t)\), and these operations are bijections. Computing the inverse of these paths
yields the desired equalities.
\end{proof}
\begin{remark}\newcommand{\simstar}{\ensuremath{\overset{*}{\sim}}}
	The duality in Proposition \ref{prop:general-good-path-duality} can now be phrased in the following manner: let \(\simstar\) denote the
	relation given by \(c \simstar z\) iff \(c\in\frakC_{-1}(z)\). We now observe that (a) in the above theorem says that \(\simstar\) is a
	symmetric relation.
\end{remark}

Now, the question to ask about is the following: where exactly does this duality come from? From the proof of Corollary \ref{corollary:S-vee-equals-S-ker},
we know that good paths are intimately related to answer questions like the ones we consider in this document. Here we have a hint:
in, say, Question \ref{question:local-infimum-question}, there is a certain \emph{parity} to the quantity of interest. That is, we considered
things of the form \(\|UAU^{-1}\|\), rather than \(\|U^{-1}AU\|\). A priori, this may seem like a trivial difference, but the two versions are actually
different; in fact, them being different yet obviously similar suggests that there should be a \emph{duality} between them. Since we now have
two sources of duality, it is sensible to conjecture that they are the same. To prove this, we will first provide a classification, as we did in the
last section, of the \(\|U^{-1}AU\|\) version of the problem.

\begin{definition}
	Let \(V\) be a Banach algebra. Define the function
	\[ \frakN_\cup^*\!: V\times V\to\R\cup\{\infty\} \]
	by
	\[ \frakN_\cup^*(a,z) := \limsup_{r\to 0}\inf_{\substack{u\in V^\times \\ \|u - z\| < r}} \| u^{-1}au \| \]
	and the algebra
	\[ \frakS_\cup^*(z) := \{ a\in V \mid \frakN_\cup^*(a,z) < \infty \}. \]
\end{definition}

Thus, we are now asking about when \(\frakN_\cup^*(a,z) < \infty\), rather than \(\frakN_\cup(a,z)\). In order to implement the proof
strategy of the last section, we need dual versions of the sets \(\frakS_C(Z)\) and \(\frakS_{\ker}(Z)\), and we need a dual
version of the kernel criterion, Proposition \ref{prop:local-erik}.

\begin{definition}
	Let \(V\) be a Banach algebra, and let \(c,z\in V\). Define the algebra
	\[ \frakS_c^*(z) := \{ a\in V \mid caz = 0 \}. \]
	If \(X\) is a Banach space and \(Z\in \sfB(X)\), define the algebra
	\[ \frakS_{\img}(Z) := \{ A\in \sfB(X) \mid A\img{Z} \subseteq \img{Z} \}. \]
\end{definition}

The above give the desired dual versions of \(\frakS_C(Z)\) and \(\frakS_{\ker}(Z)\). In order to confirm that this is the case, we prove
a dual version of the kernel criterion, which in this instance is an \emph{image} criterion instead.

\begin{lemma}\label{lemma:image-criterion}
	Let \(X\) be a Banach space, let \(A,Z\in \sfB(X)\). Assume that \(A\img(Z)\not\subseteq\img(Z)\).
	Then \(\frakN_\cup^*(A,Z) = \infty\). In particular, we have the inclusion \(\frakS_\cup^*(Z) \subseteq \frakS_{\img}(Z)\).
\end{lemma}
\begin{proof}
Suppose \(x\in\img(Z)\), so \(x = Zy\), but \(Ax\not\in\img(Z)\). Then
\begin{align*}
	\frakN_\cup^*(A,Z) = \liminf_{r\to 0}\inf_{\substack{U\in\sfB(X)^\times \\ \|U - Z\| < r}}\sup_{w\not=0}\frac{\|U^{-1}Aw\|}{\|U^{-1}w\|}
	&\geq \liminf_{r\to 0}\inf_{\substack{U\in\sfB(X)^\times \\ \|U - Z\| < r}}\frac{\|U^{-1}Ax\|}{\|U^{-1}x\|} \\
	&= \liminf_{r\to 0}\inf_{\substack{U\in\sfB(X)^\times \\ \|U - Z\| < r}}\frac{\|U^{-1}Ax\|}{\|U^{-1}Zy\|} \\
	&= \liminf_{r\to 0}\inf_{\substack{U\in\sfB(X)^\times \\ \|U - Z\| < r}}\frac{\|U^{-1}Ax\|}{\|U^{-1}Uy\|} \\ 
	&= \frac{1}{\|y\|} \liminf_{r\to 0}\inf_{\substack{U\in\sfB(X)^\times \\ \|U - Z\| < r}}\|U^{-1}Ax\| = \infty
\end{align*}
as desired.
\end{proof}

With these tools available to us, we now just step through exactly the same proof as in the last section, except dualized.

\begin{proposition}
	Let \(X\) be a Banach space, and let \(Z,C\in\sfB(X)\) be such that \(ZC = CZ = 0\). Then \(\frakS_{\img}(Z)\subseteq\frakS_C^*(Z)\).
\end{proposition}
\begin{proof}
Saying that \(CZ = 0\) is the same as saying that \(\img(Z)\subseteq\ker(C)\). Now, if \(A\in\frakS_{\img}(Z)\), then \(A\) keeps the image of \(Z\)
invariant, and hence
\[ A\img(Z)\subseteq \img(Z) \subseteq \ker(C) \implies CAZ = 0 \]
as desired.
\end{proof}
\begin{lemma}
	Let \(V\) be a Banach algebra, let \(z\in V\), and let \(z+te(t)\) be a good path with inverse \(c_{-1}t^{-1}+c(t)\). Then we have the inclusion \(\frakS_{c_{-1}}^*(z)\subseteq\frakS_\cup^*(z)\).
\end{lemma}
\begin{proof}
Suppose \(a\in\frakS_{c_{-1}}^*(z)\). Noting that
\[ \frakN^*_\cup(a,z) = \limsup_{r\to 0}\inf_{\substack{u\in V^\times \\ \|u - z\| < r}} \|u^{-1}au\| \leq \lim_{t\to 0}\| (z+te(t))^{-1}a(z+te(t)) \| \]
we compute
\begin{align*}
	(c_{-1}t^{-1} + c_0 + O(t))a(z+te_0 + O(t^2)) &= c_{-1} a zt^{-1} + c_{-1} a e_0 + c_0 az + O(t) \\
	&= c_{-1} a e_0 + c_0 az + O(t)
\end{align*}
where we used that \(c_{-1}az = 0\) by assumption. This is bounded in norm as \(t\to 0\), so we are done.
\end{proof}
\begin{theorem}\label{thm:dual-local-classification}
	Let \(X\) be a Banach space, let \(Z\in\End(\C^n)\) be such that \(\frakC_{-1}(Z)\) is non-empty, and let \(C\in\frakC_{-1}(Z)\). Then
	\[ \frakS_{\img}(Z) = \frakS^*_{C}(Z) = \frakS_\cup^*(Z). \]
\end{theorem}
\begin{proof}
Combining the above results, we have that
\[ \frakS_{\img}(Z)\subseteq\frakS^*_{C}(Z)\subseteq\frakS_\cup^*(Z)\subseteq\frakS_{\img}(Z) \]
which yields the first part.
\end{proof}

Once again, for this result to be effective practically at least in the finite-dimensional case, we need to be able to calculate (the dimension of) \(\frakS_{\img}(Z)\). Luckily for us,
we do not need to work nearly as hard to do this, because we already did it for \(\frakS_{\ker}(Z)\).

\begin{proposition}\label{prop:Simg-dimension-calculation}
	Let \(Z\in\End(\C^n)\), and let \(k = \dim\ker{Z}\). Then
	\[ \dim{\frakS_{\img}(Z)} = n^2 - kn + k^2. \]
\end{proposition}
\begin{proof}
Let \(P_Z\) be the orthogonal projection onto \(\ker(Z)\). Note that \(\ker(P_Z) = \img(Z)\), and therefore
\[ \frakS_{\img}(Z) = \frakS_{\ker}(P_Z). \]
Thus, we know that \(\frakS_{\img}(Z)\) is an algebra, and that
\[ \dim{\frakS_{\img}(Z)} = \dim{\frakS_{\ker}(P_Z)} = n^2 - mn + m^2, \]
where \(m=\dim\img(P_Z)\), by Proposition \ref{prop:Sker(Dnm)-calculation}. By our choice of \(P_Z\), we have that \(\img(P_Z) = \ker(Z)\), so \(m=k\).
\end{proof}

With the classification done, we move to the first genuine ``duality-type'' theorem.

\begin{lemma}\label{lemma:S_C=S_Z*}
	Let \(V\) be a Banach algebra, and let \(c,z\in V\). Then \( \frakS_{c}(z) = \frakS_z^*(c) \).
\end{lemma}
\begin{proof}
This is a consequence of noticing the tautology that
\[ \frakS_{c}(z) = \{a\in V\mid zac = 0\} = \frakS_z^*(c) \]
as desired.
\end{proof}

\begin{theorem}\label{thm:good-paths-give-vee=vee*}
	Let \(X\) be a Banach space, and let \(Z\in\sfB(X)\) be such that \(\frakC_{-1}(Z)\) is non-empty. Then, for any \(C\in\frakC_{-1}(Z)\),
	\[ \frakS_\cup(Z) = \frakS_\cup^*(C). \]
	In particular, if \(Z+tE(t)\) is a good path to \(Z\) with inverse \(C_{-1}t^{-1}+C(t)\), then \(\frakS_\cup(Z) = \frakS_\cup^*(C_{-1})\).
\end{theorem}
\begin{proof}
Using Theorem \ref{thm:dual-local-classification} and Lemma \ref{lemma:S_C=S_Z*}, and part (a) of Theorem \ref{thm:properties-of-C1(Z)}, note that
\[ \frakS_\cup(Z) = \frakS_C(Z) = \frakS_Z^*(C) = \frakS_\cup^*(Z). \]
The final assertion follows by letting \(C = C_{-1}\).
\end{proof}

With the above theorem in mind, we might ask the following question: for which pairs \((Z,Z')\) are the sets \(\frakS_\cup(Z)\) and \(\frakS_\cup^*(Z')\) exactly the same?
In other words, supposing one asks to find all \(A\) such that \(\frakN_\cup(A,Z)<\infty\), which \(Z'\) can be used to transform this to equivalently asking
\(\frakN_\cup^*(A,Z')\)? Due to Theorem \ref{thm:good-paths-give-vee=vee*}, we know that choosing \(Z'\in\frakC_{-1}(Z)\) works, but are there other choices? Building upon the below foundations, which give us a hint
at an algebraic classification of \(\frakC_{-1}(Z)\), we will see in Section \ref{subsection:good-paths-in-hilbert-spaces-gen-ind-zero} that there is a pleasing answer to this (see Corollary \ref{corollary:fun-suggestive-but-useless-form}).

\begin{lemma}\label{lemma:same-invariant-matrices-imply-equal}
	Let \(X\) be a Hilbert space, and let \(U,U'\subseteq X\) be subspaces of \(X\). Suppose that
	\[ \{ A\in\sfB(X) \mid AU\subseteq U \} \subseteq \{ A\in\sfB(X) \mid AU' \subseteq U' \}. \]
	Then \( U \subseteq U' \). In particular, if we have equality above, then \(U = U'\).
\end{lemma}
\begin{proof}
Let \(x\in U'\), let \(y\in U\), let \(P_x\!:X\to \Span(x)\subseteq U\) be the projection onto \(\Span(x)\). Define the map \(T\!:X\to\Span(y)\subseteq U\) as the composition
\begin{diagram*}
	\phantom{.}X\ar[r,"P_x"] & \Span(x)\ar[r,"x\, \mapsto\, y"] & \Span(y).
\end{diagram*}
Then \(TU\subseteq\Span(y)\subseteq U\), and hence \(TU'\subseteq U'\). However, since \(T(x) = y\), we see that \(y\in U'\). Since \(y\) was
arbitrary, we conclude that \(U\subseteq U'\).
\end{proof}

\begin{theorem}\label{thm:vee=vee*-iff-img=ker}
	Let \(X\) be a Banach space, and let \(C,Z\in\sfB(X)\).
	\begin{enumerate}[label=(\alph*)]
	\item Suppose \(X\) is a Hilbert space. If \(\frakS_\cup(Z) = \frakS_\cup^*(C)\), then \(\img(C) = \ker(Z)\).
	\item Suppose \(\frakC_{-1}(Z)\) and \(\frakC_{-1}(C)\) are non-empty. If \(\img(C) = \ker(Z)\), then \(\frakS_\cup(Z) = \frakS_\cup^*(C)\).
	\end{enumerate}
\end{theorem}
\begin{proof}
(a) If \(\frakS_\cup(Z)=\frakS_\cup^*(C)\), then, in particular, we know that
\[ \frakS_{\ker}(Z) = \frakS_\cup(Z) = \frakS_\cup^*(C) = \frakS_{\img}(C). \]
Applying Lemma \ref{lemma:same-invariant-matrices-imply-equal} with \(X = \C^n\), \(U=\ker(Z)\), and \(U'=\img(C)\) yields the result.

(b) Note that, by Theorems \ref{thm:good-linear-paths-equals-ker-equals-vee} \& \ref{thm:dual-local-classification}, we have
\begin{align*}
	\frakS_\cup(Z) = \frakS_{\ker}(Z) &= \{ A\in\End(\C^n) \mid A\ker(Z)\subseteq \ker(Z) \} \\
	&= \{ A\in\End(\C^n) \mid A\img(C)\subseteq \img(C) \} = \frakS_{\img}(C) = \frakS_\cup^*(C)
\end{align*}
which finishes the proof.
\end{proof}

\begin{corollary}\label{corollary:C-in-C1Z-implies-imgC=kerZ}
	Let \(X\) be a Hilbert space, let \(Z\in\sfB(X)\), and let \(C\in\frakC_{-1}(Z)\). Then
	\[ \img(C) = \ker(Z) \quad \text{and} \quad \ker(C) = \img(Z). \]
\end{corollary}
\begin{proof}
Combining the above Theorem \ref{thm:good-paths-give-vee=vee*} and Theorem \ref{thm:vee=vee*-iff-img=ker} yields the first statement. The second follows
from noting part (a) of Theorem \ref{thm:properties-of-C1(Z)}.
\end{proof}
\begin{remark}
	The above proof strategy will work in any situation where one has an analogue of Lemma \ref{lemma:same-invariant-matrices-imply-equal}.
\end{remark}

\subsection{Good Paths in Hilbert Spaces \& Operators of Generalized Index Zero}\label{subsection:good-paths-in-hilbert-spaces-gen-ind-zero}
In order to apply Theorem \ref{thm:good-linear-paths-equals-ker-equals-vee}, we must establish the existence of good paths. This is a priori very difficult, as there is seemingly nothing
to grasp onto in order to do such a thing. Furthermore, unlike in the finite-dimensional case, approximating an operator by invertible ones---which is implicit in the existence of
good paths---is no longer generally possible (for example, the left shift on \(\ell^2\) cannot be the limit of invertible operators), so one cannot expect good paths to exist for all operators.

Fortuituously, we know that in the setting of a Hilbert space, every element of \(\frakC_{-1}(Z)\) satisfies an algebraic property, provided by Corollary \ref{corollary:C-in-C1Z-implies-imgC=kerZ}:
\begin{definition}\label{definition:frakC'}
	Let \(X\) be a Banach space, and let \(Z\in \sf(X)\). Define the set
	\[ \frakC'(Z) := \{ C\in \sfB(X) \mid \img(C) = \ker(Z)\text{ and } \img(Z) = \ker(C) \}. \]
\end{definition}
\noindent Corollary \ref{corollary:C-in-C1Z-implies-imgC=kerZ} tells us that \(\frakC_{-1}(Z)\subseteq\frakC'(Z)\). Furthermore, the
definition of \(\frakC'(Z)\) provides us with a natural obstruction to consider: if we want \(\frakC'(Z)\) to be non-empty, i.e.\ for there to exist an operator \(C\) such that
\(\img(C) = \ker(Z)\) and \(\ker(C) = \img(Z)\), then we must in particular also have an isomorphism
\[ \coker(Z) = X/\img(Z) = X/\ker(C) \iso \img(C) = \ker(Z). \]
Heuristically, this is saying that \(Z\) is ``index zero'' although without the finiteness assumptions. With this in mind, we may hope that satisfying this kind of condition
removes any pathologies that may appear.

\begin{definition}
	Let \(X\) be a Banach space, and let \(Z\in\sfB(X)\). We say \(Z\) is of \emph{generalized index zero} if there exists a continuous isomorphism \(\ker(Z)\iso\coker(Z)\).
\end{definition}

\begin{remark}
	Let \(Z\in\sfB(X)\) be of index zero. Then \(Z\) is also of generalized index zero: indeed, since both \(\ker(Z)\) and \(\coker(Z)\) are finite-dimensional vector spaces of equal dimension,
	one may find a (necessarily continuous) isomorphism \(\ker(Z)\iso\coker(Z)\). In particular, if \(X\) is finite-dimensional, then every \(Z\in\sfB(X)\) is of generalized index zero.
\end{remark}

\begin{remark}\label{remark:gen-ind-zero-has-closed-image}
	If \(Z\in\sfB(X)\) is of generalized index zero, then \(Z\) has closed image. Indeed, the quotient \(\coker(Z) = X/\img(Z)\) is a Banach space if and only if \(Z\) has closed image,
	and since \(\ker(Z)\) has a continuous isomorphism to \(\coker(Z)\), the latter must be a Banach space by a generalization of the open mapping theorem. In particular, it follows
	that the inverse map \(\coker(Z)\iso\ker(Z)\) is also continuous. Similarly, it follows that \(Z\) is of generalized index zero if and only if it has closed image and there is
	a continuous isomorphism \(\coker(Z)\to\ker(Z)\).
\end{remark}

\begin{theorem}\label{thm:generalized-index-zero-iff-C'-non-empty}
	Let \(X\) be a Banach space, and let \(Z\in\sfB(X)\). Then there is a bijection
	\[ \frakC'(Z) \cong \left\{\, \text{continuous isomorphisms } \ker(Z)\iso\coker(Z) \,\right\} \]
	which associates to \(C\in\frakC'(Z)\) the isomorphism \(\theta_C\) given by
	\[ \ker(Z) = \img(C) \iso X/\ker(C) = X/\img(Z) = \coker(Z) . \]
	In particular, \(Z\) is of generalized index zero if and only if \(\frakC'(Z)\) is non-empty.
\end{theorem}
\begin{proof}
We have already provided a map
\[ \frakC'(Z) \to \left\{\, \text{continuous isomorphisms } \ker(Z)\iso\coker(Z) \,\right\} \]
and so it remains to construct an inverse. Let \(\theta\!:\ker(Z)\to\coker(Z)\) be a continuous isomorphism, and define \(C_\theta\!:X\to X\) by the diagram
\begin{diagram*}
	X \ar[d,two heads]\ar[r,dashed,"C_\theta"] & X \\
	\coker(Z) \ar[r,"\theta^{-1}"] & \ker(Z)\ar[u,hook]
\end{diagram*}
Trivially, since \(C_\theta\) is a composition of continuous maps, it follows that \(C_\theta\) is continuous. Furthermore, \(\ker(C_\theta) = \ker(X\sur\coker(Z)) = \img(Z)\), and \(\img(C_\theta) = \ker(Z)\), so \(C_\theta\in\frakC'(Z)\).

We need to check that for all \(C\in\frakC'(Z)\), we have \(C_{\theta_C} = C\), and that for all continuous isomorphisms \(\theta\!:\coker(Z)\iso\ker(Z)\) we have \(\theta_{C_\theta} = \theta\).
\begin{itemize}[label=\(\star\)]
\item If \(C\in\frakC'(Z)\), by the first isomorphism theorem and that \(\img(Z) = \ker(C)\), the map \(C_{\theta_C}\) is given by \(x\mapsto \theta_C^{-1}(x + \ker(C)) = Cx\).
\item If \(\theta\!:\ker(Z)\to\coker(Z)\) is a continuous isomorphism, then \(\theta_{C_\theta}\) is given by sending \(z\in\ker(Z)\) to
\[ C_\theta^{-1}\{x\} = \theta(x). \]
\end{itemize}
This concludes the proof.
\end{proof}

In the beginning of this section, as well as in Section \ref{subsection:duality-in-coefficients}, we implicitly promised a complete algebraic description of the elements of \(\frakC_{-1}(-)\),
in particulary given by \(\frakC'(-)\). Because of Theorem \ref{thm:generalized-index-zero-iff-C'-non-empty}, proving this would also yield the other main goal of this section, namely
the existence of good paths, thereby fulfilling all our previous promises. Our approach is a fairly standard one: begin with self-adjoint operators, then lift using a slightly sharpened version of Polar decomposition.

\begin{lemma}\label{lemma:hilbert-self-adjoint-gen-index-zero-have-C1=C'}
	Let \(X\) be a Hilbert space, and let \(Z\in\sfB(X)\) be a self-adjoint operator. Then \(Z\) is of generalized index zero if and only if \(Z\) has closed image, in which case
	\[ \frakC_{-1}(Z) = \frakC'(Z) \not= \varnothing. \]
\end{lemma}
\begin{proof}
If \(Z\) is of generalized index zero, then \(Z\) has closed image, as argued in Remark \ref{remark:gen-ind-zero-has-closed-image}. Conversely, if \(Z\) has closed image, we can write
\[ X = \img(Z) \oplus \img(Z)^\perp = \img(Z)\oplus\ker(Z^*) = \img(Z)\oplus\ker(Z) \]
since \(Z\) is self-adjoint. In particular,
\[ X/\img(Z) = (\img(Z)\oplus\ker(Z))/\img(Z) \cong \ker(Z) \]
so \(Z\) is of generalized index zero.

We now prove the second half of the result. That \(\frakC'(Z)\) is non-empty follows immediately by Theorem \ref{thm:generalized-index-zero-iff-C'-non-empty}. Furthermore, Corollary \ref{corollary:C-in-C1Z-implies-imgC=kerZ} shows
that \(\frakC_{-1}(Z)\subseteq\frakC'(Z)\), so it remains only to show the reverse inclusion. For this, let \(C\in\frakC'(Z)\); one sees that, in the decomposition given above, the operators \(Z,C\) have the form
\[ Z = \begin{pmatrix} Z_{11} & 0 \\ 0 & 0 \end{pmatrix},\quad C = \begin{pmatrix} 0 & 0 \\ 0 & C_{22} \end{pmatrix} \]
where \(Z_{11}\!:\img(Z)\iso\img(Z)\) and \(C_{22}\!:\ker(Z)\iso\ker(Z)\) are continuous isomorphisms. Let \(E\!:X\to X\) and \(M\!:X\to X\) be given by
\[ E = \begin{pmatrix} 0 & 0 \\ 0 & C_{22}^{-1} \end{pmatrix},\quad M = \begin{pmatrix} Z_{11}^{-1} & 0 \\ 0 & 0 \end{pmatrix}. \]
Then, for \(t > 0\),
\[ (Z+tE)(Ct^{-1}+M) = \begin{pmatrix} Z_{11} & 0 \\ 0 & C_{22}^{-1}t \end{pmatrix} \begin{pmatrix} Z_{11}^{-1} & 0 \\ 0 & C_{22}t^{-1} \end{pmatrix} = \begin{pmatrix} \id_{\img(Z)} & 0 \\ 0 & \id_{\ker(Z)} \end{pmatrix} = \id_X \]
and similarly for \((Ct^{-1} + M)(Z+tE)\). Therefore, \(Z+tE\in\calE(Z)\) and thus \(C\in\frakC_{-1}(Z)\).
\end{proof}

\begin{lemma}\label{lemma:hilbert-space-closed-image-iff-bounded-below-on-kerperp}
	Let \(X\) be a Hilbert space, and let \(Z\in\sfB(X)\). Then \(Z\) has closed image if and only if there exists some \(c > 0\) such that \(\|Zx\| \geq c\|x\|\) for all \(x\in\ker(Z)^\perp\).
\end{lemma}
\begin{proof}
If \(Z\) has closed image, then \(\img(Z)\) is itself a Hilbert space, and the map \(Z|_{\ker(Z)^\perp}\!:\ker(Z)^\perp\to\img(Z)\) is a bounded isomorphism. By the open mapping theorem, one then
sees that there is some \(c>0\) such that \(\|Zx\| \geq c \|x\|\) for all \(x\in\ker(Z)^\perp\).

Conversely, suppose \(\|Zx\| \geq c\|x\|\) for all \(x\in\ker(Z)^\perp\), and consider a Cauchy sequence \(Zx_n\in\img(Z)\). Since we may write
\[ X = \ker(Z) \oplus \ker(Z)^\perp, \]
we may assume that \(x_n\in\ker(Z)^\perp\). Using our assumption, we see that
\[ \| x_n - x_m \| \leq \frac{1}{c} \| Zx_n - Zx_m \| \]
which implies that the sequence \(x_n\) is also Cauchy, and hence convergent to some \(x\in\ker(Z)^\perp\). By continuity, we then have \(Zx_n \to Zx\), so that \(\img(Z)\) is closed.
\end{proof}

\begin{lemma}
	Let \(X\) be a Hilbert space, and let \(Z\in\sfB(X)\). Then \(Z\) has closed image if and only if \(\sqrt{Z^*Z}\) has closed image.
\end{lemma}
\begin{proof}
Since \(\|Zx\| = \|\sqrt{Z^*Z}x\|\) for all \(x\in X\), we deduce that
\[ \|Zx\| \geq c \|x\| \iff \|\sqrt{Z^*Z}x\| \geq c \|x\| \]
which yields the result, applying Lemma \ref{lemma:hilbert-space-closed-image-iff-bounded-below-on-kerperp}.
\end{proof}

\begin{theorem}[Sharpened polar decomposition]\label{thm:polar-decomposition}
	Let \(X\) be a Hilbert space, and let \(Z\in\sfB(X)\) be of generalized index zero with associated isomorphism \(\theta\!:\ker(Z) \iso \coker(Z)\).
	Then there is some \(U\in\sfB(X)^\times\), which is also a partial isometry, such that \( Z = U\sqrt{Z^* Z} \). Furthermore, if \(\theta\) is an isometry, then \(U\) is unitary.
\end{theorem}
\begin{proof}
The operator \(\sqrt{Z^*Z}\) has closed image since \(Z\) has closed image. Furthermore, since \(\|Zx\| = \| \sqrt{Z^*Z}x \|\), we also have that \(\ker(Z) = \ker(\sqrt{Z^*Z})\). Therefore,
we obtain an isomorphism \(U^\pre\) using the first isomorphism theorem
\[ U^\pre\!: \img(\sqrt{Z^*Z}) \iso V/\ker(\sqrt{Z^*Z}) = V/\ker(Z) \iso \img(Z), \]
given explicitly by \(\sqrt{Z^*Z}x \mapsto Zx\). This is an isometry since \(\|Zx\| = \|\sqrt{Z^*Z}x\|\). We would now like to extend this to a continuous invertible operator \(X\to X\). To do this, we make use of
the fact that we have two separate decompositions for \(X\), namely
\[ X = \img(\sqrt{Z^*Z})\oplus\ker(\sqrt{Z^*Z}) = \img(\sqrt{Z^*Z})\oplus\ker(Z) \]
and
\[ \phantom{.}X = \img(Z)\oplus\img(Z)^\perp. \]
For notational purposes, we let \(\iota\!:\coker(Z)\iso\img(Z)^\perp\) be the obvious isomorphism. From this, we define \(U\!:X\iso X\) by
\begin{diagram*}
	X \ar[r,"\sim"] & \img(\sqrt{Z^*Z})\oplus\ker(Z) \ar[r,"(U^\pre\text{,}\,\theta)"] & \img(Z)\oplus\coker(Z) \ar[r,"(\id\text{,}\,\iota)"] & \img(Z)\oplus\img(Z)^\perp\ar[r,"\sim"] & X
\end{diagram*}
where we note that every map involved is a continuous isomorphism, so that \(U\in\sfB(X)^\times\). It is now clear by the first decomposition that for all \(x = x_0 + \sqrt{Z^*Z}x_1\in X\), we will have
\[ Zx = Z\sqrt{Z^*Z}x_1 = U^\pre\sqrt{Z^*Z}\sqrt{Z^*Z}x_1 = U^\pre\sqrt{Z^*Z}(x_0 + \sqrt{Z^*Z}x_1) = U\sqrt{Z^*Z}x \]
so that \(Z=U\sqrt{Z^*Z}\). Finally, when \(\theta\) above is an isometry, we see that \(U\) is a surjective isometry, hence unitary.
\end{proof}

\begin{corollary}\label{corollary:gen-ind-zero-iff-closed-image-and-almost-self-adjoint}
	Let \(X\) be a Hilbert space, and let \(Z\in\sfB(X)\). Then \(Z\) is of generalized index zero if and only if \(\img(Z)\) is closed and there is some invertible operator \(U\in\sfB(X)^\times\) such
	that \(UZ\) is self-adjoint.
\end{corollary}
\begin{proof}
If \(Z\) is of generalized index zero, then this is a straightforward consequence of Remark \ref{remark:gen-ind-zero-has-closed-image} and Theorem \ref{thm:polar-decomposition}. Conversely,
if \(Z\) has closed image and we are given some \(U\in\sfB(X)^\times\) such that \(UZ\) is self-adjoint, then we see that \(UZ\) is of generalized index zero by Lemma \ref{lemma:hilbert-self-adjoint-gen-index-zero-have-C1=C'}.
Since \(\img(UZ) = U\img(Z)\cong\img(Z)\) and \(\ker(UZ) = \ker(Z)\), we see that \(Z\) is of generalized index zero.
\end{proof}

\begin{lemma}\label{lemma:gen-index-zero-implies-sqrt-gen-index-zero}
	Let \(X\) be a Hilbert space, and let \(Z\in\sfB(X)\). If \(Z\) is of generalized index zero, then so is \(\sqrt{Z^*Z}\).
\end{lemma}
\begin{proof}
By the proof of Theorem \ref{thm:polar-decomposition}, we have an isometric isomorphism
\[ U^\pre\!:\img(\sqrt{Z^*Z})\to\img(Z),\]
and thus we have an induced isomorphism given by the composite
\begin{diagram*}
	\ker(\sqrt{Z^*Z}) \ar[r,equal] & \ker(Z)\ar[r,"\sim"] & \coker(Z) \ar[r,"\sim"] & \coker(\sqrt{Z^*Z})
\end{diagram*}
as desired.
\end{proof}

\begin{theorem}\label{thm:hilbert-space-gen-index-zero-C1=C'}
	Let \(X\) be a Hilbert space, and let \(Z\in\sfB(X)\) be of generalized index zero. Then
	\[ \frakC_{-1}(Z) = \frakC'(Z) \not= \varnothing. \]
\end{theorem}
\begin{proof}
Using Theorem \ref{thm:polar-decomposition}, write \(Z = U\sqrt{Z^*Z}\). By Lemma \ref{lemma:gen-index-zero-implies-sqrt-gen-index-zero}, \(\sqrt{Z^*Z}\) is of generalized index zero. Furthermore,
since it is self-adjoint, Lemma \ref{lemma:hilbert-self-adjoint-gen-index-zero-have-C1=C'} means we can compute
\[ \frakC_{-1}(Z) = \frakC_{-1}(U\sqrt{Z^*Z}) = \frakC_{-1}(\sqrt{Z^*Z})U^{-1} = \frakC'(\sqrt{Z^*Z})U^{-1} = \frakC'(U\sqrt{Z^*Z}) = \frakC'(Z). \]
This completes the proof.
\end{proof}
\begin{corollary}\label{corollary:hilbert-space-S_V=S_ker}
	Let \(X\) be a Hilbert space, and let \(Z\in\sfB(X)\) be of generalized index zero. Then
	\[ \frakS_\cup(Z) = \frakS_{\ker}(Z) = \frakS_C(Z) \]
	where \(C\in\frakC_{-1}(Z)\not=\varnothing\).
\end{corollary}
\begin{proof}
Apply Theorem \ref{thm:good-linear-paths-equals-ker-equals-vee} and Theorem \ref{thm:hilbert-space-gen-index-zero-C1=C'}.
\end{proof}

Notably, we can rather remarkably determine \(\frakC_{-1}(Z)\) entirely through an algebraic description, despite its analytic nature.
The theorem can also be realized in the following highly suggestive form:

\begin{corollary}\label{corollary:fun-suggestive-but-useless-form}
	Let \(X\) be a Hilbert space, \(Z\in\sfB(X)\) of generalized index zero. Then
	\[ \frakC_{-1}(Z) = \{ C\in\sfB(X) \mid \frakS_\cup(Z) = \frakS_\cup^*(C) \text{ and } \frakS_\cup(C) = \frakS_\cup^*(Z) \}. \]
\end{corollary}
\begin{proof}
Apply Theorem \ref{thm:hilbert-space-gen-index-zero-C1=C'} together with Theorem \ref{thm:vee=vee*-iff-img=ker}.
\end{proof}

\begin{remark}
	As a consequence of Theorem \ref{thm:hilbert-space-gen-index-zero-C1=C'}, we see that for \(C,Z\in\End(\C^n)\), we have \(C\in\frakC_{-1}(Z)\) if and only if the sequence
	\begin{diagram*}
		\cdots\ar[r,"C"] & \C^n\ar[r,"Z"] & \C^n\ar[r,"C"] & \C^n\ar[r,"Z"] & \cdots
	\end{diagram*}
	is exact. Furthermore, denoting a complex of the form
	\begin{diagram*}
		\cdots\ar[r,"X"] & \C^n\ar[r,"Z"] & \C^n\ar[r,"X"] & \C^n\ar[r,"Z"] & \cdots
	\end{diagram*}
	by \(\sfC(X)\), if \(X\in\End(\C^n)\) is such that we have a quasi-isomorphism \(\sfC(X)\cong\sfC(C)\), then \(X\in\frakC_{-1}(Z)\). Indeed,
	we would have \(\HH^k(\sfC(X))\cong\HH^k(\sfC(C))=0\) for all \(k\) since \(\sfC(C)\) is exact, and therefore \(\sfC(X)\) is also exact.
\end{remark}

\subsection{Pathwise Boundedness Along Admissible Paths}\label{subsection:pathwise-boundedness}
Ultimately, aside from computing such things as \(\frakS(-)\) and \(\frakS_\cup(-)\), we are really also interested in boundedness behaviour along \emph{specified} paths. Indeed,
it is an immediate refinement of both \(\frakS(-)\) and \(\frakS_\cup(-)\) to characterize boundedness along a given arbitrary path; understanding pathwise boundedness necessarily gives us more
information about the fuzzier questions we have been addressing above.

A valuable approach now is to try to investigate \(\frakS_\pw(Z; E(t))\) (see Definition \ref{definition:frakS_pw}) for particularly nice functions \(E(t)\), namely analytic ones, and hope that this can in some way be characterized by
an algebraic condition. To this end, we introduce yet another gadget:
\begin{definition}\label{definition:frakS-filtration}
	Let \(X\) be a Banach space, let \(V\subseteq X\) be a closed subspace, and let \(F^\bullet V\) be a filtration of \(V\)
	\[ V = F^0 V \supseteq F^1V  \supseteq F^2V \supseteq \cdots \]
	by closed subspaces. Define the algebra
	\[ \frakS_\frakF(F^\bullet V) := \{ A\in\sfB(X) \mid \forall i \geq 0,\,\, A(F^iV)\subseteq F^iV \}. \]
	To any sequence of closed subspaces \(V_k\subseteq X\), \(k\geq 0\), one may associate a filtration \(F^\bullet_{\cap V_k} V_0\) of \(V_0\) given by
	\[ F^i_{\cap V_k}V_0 := \bigcap_{k = 0}^i V_k. \]
	In the situation of such a filtration, we use the notation \(\frakS_\frakF(\cap_k V_k) := \frakS_\frakF(F^\bullet_{\cap V_k}V_0)\).
\end{definition}
\begin{notation}
	Observe that for any filtration \(F^\bullet V\) of \(V\) by closed subspaces, one can for each \(n\geq 0\) obtain a new filtration, which we denote \(F^{\bullet\leq n} V\), simply by truncating \(F^\bullet V\) after the \(n\)th term,
	i.e.\ by setting
	\[ F^{i\leq n}V = \begin{cases} F^iV & \text{if } i \leq n, \\ 0 & \text{otherwise.} \end{cases} \]
	When \(F^\bullet V = F^\bullet_{\cap_kV_k}V_0\) is the filtration of a closed subspace \(V = V_0\) induced by a sequence of closed subspaces \(V_k\), we will write
	\[ \frakS_\frakF(F^{\bullet\leq n}_{\cap_kV_k}V_0) := \frakS_\frakF(\cap_k^{\leq n}V_k). \]
\end{notation}
\begin{proposition}\label{prop:truncated-filtration}
	Let \(X\) be a Banach space, let \(V\subseteq X\) be a closed subspace, and let \(F^\bullet V\) be a filtration of \(V\) by closed subspaces. Then, for all \(0 \leq m \leq n \leq \infty\),
	\[ \frakS_\frakF(F^{\bullet\leq n} V) \subseteq \frakS_\frakF(F^{\bullet\leq m}V). \]
	In particular, choosing \(n=\infty\), \(m=0\), and \(V = \ker(Z)\) for some \(Z\in\sfB(X)\), we have
	\[ \frakS_\frakF(F^\bullet\ker(Z))\subseteq\frakS_{\ker}(Z).\]
\end{proposition}
\begin{proof}
Follows immediately from the definitions.
\end{proof}

\begin{proposition}
	Let \(X\) be a Banach space, let \(E_k\in\sfB(X)\) for all \(k\geq 0\) be such that \(\sum_{k=0}^\infty E_kt^k\) is an admissible path converging to \(E_0\). Then there
	exists some \(0\leq n \leq \infty\) such that
	\[ \bigcap_{k=0}^n\ker{E_k} = 0. \]
	In particular, for all \(m\geq n\) (including \(m=\infty\)), we have
	\[ \frakS_\frakF(\cap_k^{\leq m}\ker{E_k}) = \frakS_{\frakF}(\cap_k^{\leq n}\ker{E_k}). \]
	If \(X\) is finite-dimensional, then one may choose \(n < \infty\).
\end{proposition}
\begin{proof}
If the given intersection is non-trivial, then the path could not be invertible---which is required, locally to be admissible---hence yielding a contradiction. The second
assertion follows by noting that all additional conditions imposed in higher truncations are trivial, as
\[ \bigcap_{k=0}^{m}\ker{E_k} \subseteq \bigcap_{k=0}^{n}\ker{E_k} = 0. \]
That \(n\) can be chosen to be finite when \(X\) is a finite-dimensional space is clear, as any strictly descending chain of subspaces will stabilize at some point.
\end{proof}

\begin{proposition}\label{prop:analytic-paths-boundedness-criterion}
	Let \(X\) be a Banach space, let \(E_k\in\sfB(X)\), \(k \geq 0\), and consider an admissible path \(\sum_{k=0}^\infty E_kt^k\), together with its induced
	sequence of subspaces \(\ker(E_k)\), \(k \geq 0\). Then
	\[ \frakS_{\pw}\left(E_0; \sum_{k=0}^\infty E_kt^k \right) \subseteq \frakS_\frakF(\cap_k\ker(E_k)). \]
\end{proposition}
\begin{proof}
Let \(E(t) = \sum_kE_kt^k\). If \(A\in\frakS_\pw(E_0; E(t))\), then, per the definition, we have
\[ \lim_{t\to 0}\|E(t)AE(t)^{-1}\| = \lim_{t\to 0}\sup_{x\in X}\frac{\|E(t)Ax\|}{\|E(t)x\|} < \infty \]
and thus, in particular,
\[ \lim_{t\to 0}\sup_{x\in \ker{E_0}}\frac{\|E(t)Ax\|}{\|E(t)x\|} \leq \lim_{t\to 0}\sup_{x\in X}\frac{\|E(t)Ax\|}{\|E(t)x\|} < \infty. \]
For notational simplicity, we will omit the limit over \(t\). Expanding the above out and using that the left supremum is over \(\ker{E_0}\), we have
\[ \sup_{x\in \ker{E_0}}\frac{\|E_0Ax + tE_1Ax + t^2E_2 Ax + \cdots + t^kE_kAx + \cdots\|}{\|tE_1x + t^2E_2 x + \cdots + t^kE_kx + \cdots\|} < \infty \]
which implies that we necessarily have \(E_0Ax = 0\) when \(x\in\ker{E_0}\), i.e.\ that \(A\ker{E_0}\subseteq\ker{E_0}\). However, using that fact and factoring out the \(t\), we obtain
\[ \sup_{x\in \ker{E_0}}\frac{\|E_1Ax + tE_2 Ax + \cdots + t^{k-1}E_kAx + \cdots\|}{\|E_1x + tE_2 x + \cdots + t^{k-1}E_kx + \cdots\|} < \infty. \]
We now proceed by induction: taking a supremum over \(\ker(E_0)\) will be greater than taking a supremum over \(\ker(E_0)\cap\ker(E_1)\), and so on. In particular, repeatedly applying the above reasoning,
we see that, for each \(j \geq 0\),
\[ A\left(\bigcap_{k=0}^j\ker(E_k)\right) \subseteq \ker(E_j). \]
For all \(0 \leq j' \leq j\) we have
\[ \bigcap_{k=0}^j\ker(E_k) \subseteq \bigcap_{k=0}^{j'}\ker(E_k) \implies A\left(\bigcap_{k=0}^{j}\ker(E_k)\right) \subseteq A\left(\bigcap_{k=0}^{j'}\ker(E_k)\right) \subseteq \ker(E_{j'}) \]
and thus
\[ A\left(\bigcap_{k=0}^{j}\ker(E_k)\right) \subseteq \bigcap_{k=0}^{j}\ker(E_k) \]
which is exactly the statement that \(A\in\frakS_\frakF(\cap_k\ker{E_k})\).
\end{proof}
\begin{remark}
	In the case where the path is linear, i.e.\ of the form \(Z+tE\), note that invertibility forces \(\ker(Z)\cap\ker(E) = 0\). Applying the above proposition, one just obtains the kernel criterion \(\frakS_{\ker}(Z)\)
	on the right-hand side, since the filtration will be given by
	\[ \ker(Z)\supseteq \ker(Z)\cap\ker(E) = 0 \supseteq 0 \cdots. \]
	In other words, the proposition suggests that boundedness along a linear path is strongly related to the kernel criterion of \(\frakS_{\ker}(Z)\), which aligns with the results of Section \ref{subsection:bounded-unmodified-existence},
	particularly in light of the below Theorem \ref{thm:good-path-truncation-rigidity}.
\end{remark}

Proposition \ref{prop:analytic-paths-boundedness-criterion} turns out to imply following \emph{striking} result about the rigidity of good paths, which, following the above remark, suggests that
good paths are fundamentally controlled by linear phenomena.
\begin{theorem}\label{thm:good-path-truncation-rigidity}
	Let \(X\) be a Banach space, and let \(E_0\in\sfB(X)\) be such that there exists a good path \(E_0+tE(t) = \sum_{k=0}^\infty E_kt^k\), with corresponding operator \(C_{-1}\in\frakC_{-1}(E_0)\). Then
	\[ \frakS_\pw(E_0;tE(t)) = \frakS_{C_{-1}}(E_0) = \frakS_{\frakF}(\cap_k\ker{E_k}) = \frakS_{\ker}(E_0). \]
	In particular, for all \(0 \leq m \leq n \leq \infty\), we have \( \frakS_\frakF(\cap_k^{\leq n}\ker(E_k)) = \frakS_\frakF(\cap_k^{\leq m}\ker(E_k)) \).
\end{theorem}
\begin{proof}
By Proposition \ref{prop:truncated-filtration}, Proposition \ref{prop:analytic-paths-boundedness-criterion}, the results of Section \ref{subsection:bounded-unmodified-existence}, and the remarks above, we have
\[ \frakS_{\ker}(E_0) = \frakS_{C_{-1}}(E_0) \subseteq \frakS_\pw(E_0;tE(t)) \subseteq \frakS_\frakF(\cap_k\ker{E_k})\subseteq\frakS_{\ker}(E_0) \]
yielding the first result. The second follows by noting that, for all \(m,n\) as stated, we have
\[ \frakS_\frakF(\cap_k^{\leq n}\ker(E_k)) \subseteq \frakS_\frakF(\cap_k^{\leq n}\ker(E_k)) \subseteq \frakS_\frakF(\cap_k^{\leq m}\ker(E_k)) \subseteq
\frakS_\frakF(\cap_k^{\leq 0}\ker(E_k)) = \frakS_{\ker}(E_0) \]
providing us the desired equalities.
\end{proof}
\begin{corollary}\label{corollary:good-path-rigidity-in-kernels}
	Let \(X\) be a Banach space, and let \(E_0\in\sfB(X)\) be such that there exists a good path \(E_0+tE(t) = \sum_{k=0}^\infty E_kt^k\). Then there exists some $n > 0$ such that
	\[ \forall m < n,\, \ker{E_0}\subseteq\ker{E_m},\quad\text{and}\quad \ker{E_0}\cap\ker{E_n} = 0. \]
\end{corollary}
\begin{proof}
Apply the second statement in Theorem \ref{thm:good-path-truncation-rigidity}. In particular, since
\[ \frakS_\frakF(\cap_k^{\leq 1}\ker(E_k)) = \frakS_\frakF(\cap_k^{\leq 0}\ker(E_k)) = \frakS_{\ker}(E_0), \]
in order to avoid a strictly non-trivial condition, we must either have \(\ker{E_0}\subseteq \ker{E_1}\), or we must have \(\ker{E_0}\cap\ker{E_1} = 0\). In the latter case, we are done; in the former,
we repeat the same reasoning. Eventually, this must terminate: if it did not, we would have \(\ker(E_0)\subseteq \ker(E_k)\) for every \(E_k\), in which case the path could not be invertible.
\end{proof}
\begin{remark}
	Consider two paths \(U_k^{(1)} \to Z\) and \(U_k^{(2)}\to Z\). Then one may form a new path to \(Z\) given by
	\[ (U_k^{(1,2)})_{k\geq 1} := (U_1^{(1)}, U_1^{(2)}, U_2^{(1)}, U_2^{(2)}, \ldots, U_k^{(1)}, U_k^{(2)},\ldots), \]
	that is, by intertwining the paths given by \(U_k^{(1)}\) and \(U_k^{(2)}\). The pathwise algebra \(\frakS_{\pw}(Z; U_k^{(1,2)})\) then acts like taking the intersection between \(\frakS_\pw(Z;U_k^{(1)})\) and
	\(\frakS_\pw(Z;U_k^{(2)})\). Furthermore, the same general trick can be done with a countable sequence of paths \(U_k^{(n)}\), \(n \geq 1\), by considering the path
	\[ U_k^{(\geq 0)} := (U_1^{(1)}, U_1^{(2)}, U_2^{(1)}, U_1^{(3)}, U_2^{(2)}, U_3^{(1)}, U_1^{(4)}, U_2^{(3)}, U_3^{(2)}, U_4^{(1)}, \ldots). \]
\end{remark}
\begin{remark}
	While it is generally hard to produce elements of \(\frakS_\pw(Z;U_k)\) for some \(U_k\to Z\), there are certain simple conditions that can be used to produce examples, at least in principle. Notably,
	if \(A\) satisfies \(U_n A = B U_n\), then \(A \in \frakS_\pw(Z;U_k)\); in particular, then one simply has that \(U_nAU_n^{-1} = B\) is constant.
\end{remark}
\begin{remark}
	There is a more general approach one can use to determine \(\frakS^\varphi_{\pw}(Z;U_k)\) in some cases. Let \(X\) be a Banach space, let \(\varphi\!:\sfB(X)\to\sfB(X)\) be a bounded linear operator,
	and let \(U\in\sfB(X)^\times\) be an invertible operator. To this, we can associate a bounded operator
	\[ \Omega_{\varphi,U}\!:\sfB(X)\to\sfB(X),\quad  A\mapsto \varphi(UAU^{-1}). \]
	In particular, if \(U_k\to Z\) is some convergent sequence of invertible operators on \(X\), then one obtains from this a sequence of bounded operators \(\Omega^{(k)} := \Omega_{\varphi,U_k}\)
	on \(\sfB(X)\). Taking the limit, one obtains an \emph{unbounded} operator \(\Omega\) on \(\sfB(X)\), and the task of determining \(\frakS^\varphi_{\pw}(Z;U_k)\) is exactly that of determining the
	domain of \(\Omega\).

	If \(X\) is \emph{finite-dimensional,} this can in principle be done computationally on a case-by-case basis. Picking a basis on \(X\cong\C^n\), one obtains a basis on \(\sfB(X)\cong\End(\C^n)\)
	given by the elementary matrices, and the operators \(\Omega^{(k)}\) can be expressed as \(n^2\times n^2\) matrices under this basis, with entries \(\omega^{(k)}_{ij}\). By the equivalence of norms in finite-dimensional spaces, we
	can assume that the norm on \(\sfB(X)\) is the \(L^{\infty}\)-norm on \(\C^{n^2}\), in which case
	\[ \lim_{k\to\infty}\|\Omega^{(k)}A\| < \infty \iff \lim_{k\to\infty} \max_{1\leq i \leq n^2} \left| \sum_{j=1}^{n^2}\omega^{(k)}_{ij}A_{(j)} \right| < \infty \]
	where by \(A_{(j)}\) we mean the \(j\)th coefficient of \(A\) as a vector in \(\sfB(X)\cong\C^{n^2}\). When everything is explicit, this is something which could be checked on a computer, as it corresponds
	to what is essentially a system of linear constraints.
\end{remark}
\begin{remark}
	If one constrains the situation even more, to the case of a finite-dimensional \(X\) and a \emph{polynomial} path \(Z+tE(t)\) (i.e.\ where \(E(t)\) is a polynomial in \(t\) with coefficients in \(\sfB(X)\cong\End(\C^n)\)),
	then one can characterize \(\frakS^\varphi_{\pw}(Z;Z+tE(t))\) in a yet simpler way. In general, for any path \(Z+tE(t)\), since for all (sufficiently small) \(t > 0\) the matrix \(Z+tE(t)\) is invertible, we can observe that
	\[ \adj(Z+tE(t))\cdot (Z+tE(t)) = \det(Z+tE(t))\cdot I \implies (Z+tE(t))^{-1} = \frac{\adj(Z+tE(t))}{\det(Z+tE(t))}  \]
	where \(\adj(-)\) denotes the adjucate matrix. In particular,
	\[ \lim_{t\to 0}\| (Z+tE(t)) A (Z+tE(t))^{-1} \| = \lim_{t\to 0} \frac{ \| (Z+tE(t))A\adj(Z+tE(t)) \| }{\left|\det(Z+tE(t))\right|} < \infty \]
	happens if and only if
	\[ \| (Z+tE(t))A\adj(Z+tE(t)) \| = O(\det(Z+tE(t)))\quad \text{as } t\to 0. \]
	When \(E(t)\) is chosen to be a polynomial in \(t\) (and the norm is chosen appropriately), both sides of the above will be polynomials in \(t\), and thus comparing
	growth rates comes down purely to comparing which one has the lowest degree power of \(t\) with a non-zero coefficient. This means that it is possible, in a situation where everything
	is explicit, to completely determine necessary and sufficient conditions for a matrix \(A\) to be bounded along the path \(Z+tE(t)\).
\end{remark}

\subsection{A Preorder on Bounded Linear Maps}\label{subsection:preorder}
In the preceding sections, we have focused our attention solely on \(\frakS_\cup(-)\). We will now turn to the harder question of understanding \(\frakS_\cup^\varphi(-)\). To help us along
with this, and to give ourselves a framework in which to place this question, we will note that
there is some interesting behaviour exhibited when \emph{comparing} different choices of \(\varphi\). This will allow us to answer certain cases of Question \ref{question:hadamard-question}.
To study this more systematically, we will now introduce a relation to encode these comparisons in a very general context. Most of this will not be necessary for our purposes.

\begin{notation}\label{notation:preorder-relation}
	Let \(V\) be a Banach algebra, and \(X\), \(Y\) be Banach spaces. Let \(\varphi\!:V\to X\) and \(\psi\!:V\to Y\) be bounded linear maps. We define the relation \(\preceq\) by \(\varphi\preceq\psi\) if and only if for all \(a\in V\),
	there exists a right-continuous order-preserving \(f_a\!:\R_{\geq 0}\to\R_{\geq 0}\) depending only on the conjugacy class of \(a\), such that \(\|\varphi(a)\|\leq f_a(\|\psi(a)\|)\). If \(\varphi \preceq \psi\)
	and \(\psi \preceq \varphi\), then we write \(\varphi\simeq\psi\).
\end{notation}

\begin{proposition}
	The relation \(\preceq\) is a preorder, i.e.\ it is reflexive and transitive. In particular, \(\simeq\) is an equivalence relation.
\end{proposition}
\begin{proof}
That \(\psi\preceq\psi\) for all \(\psi\!:V\to X\) is clear. If \(\psi \preceq \psi' \preceq \psi''\) is exhibited by the sets of functions \(\{f_a\}_a\) and \(\{g_a\}_a\), then
\[ \|\psi(a)\|\leq f_a(\|\psi'(a)\|) \leq f_a(g_a(\|\psi''(a)\|)) \]
so that \(\psi \preceq \psi''\). That \(\simeq\) is now an equivalence relations follows trivially.
\end{proof}

The below proposition will be useful, at least in principle, to reduce the complexity of our situation (primarily for finite-dimensional spaces). To be as general as possible, we introduced a relation on a very loose set of maps (and, in fact, being
a \emph{set} here is subject to some set-theoretical subtleties), but we are primarily interested in the case when \(V\) is a \emph{finite-dimensional} Banach algebra. In this case, one can reduce to considering linear endomorphisms of \(V\) itself;
see Corollary \ref{corollary:finite-dimensional-banach-algebras-only-require-endomorphisms}.

\begin{proposition}\label{prop:preceq-composition-relations-1}
	Let \(V\) be a Banach algebra, and suppose we have bounded linear maps
	\begin{diagram*}
		V \ar[r,"\varphi"] & X \ar[r, "\rho"] & X'.
	\end{diagram*}
	Then \(\rho\circ\varphi \preceq \varphi\). If, in additon, there is some constant \(C>0\) such that \(\|x\|_X \leq C\|\rho(x)\|_{X'}\) (for example, if \(\rho\) is an isometry), then \(\rho\circ\varphi \simeq \varphi\).
\end{proposition}
\begin{proof}
We have that for all \(a\in V\),
\[ \|(\rho\circ\varphi)(a)\| \leq \|\rho\|\cdot\|\varphi(a)\|. \]
Therefore, \(\rho\circ\varphi\preceq\varphi\). For the final statement, note that, given the assumptions, we have
\[ \| \varphi(a) \|_X \leq C\cdot\|(\rho\circ\varphi)(a)\| \]
which supplies the desired \(\varphi\preceq\rho\circ\varphi\).
\end{proof}
\begin{corollary}\label{corollary:finite-dimensional-banach-algebras-only-require-endomorphisms}
	Let \(V\) be a finite-dimensional Banach algebra, let \(X\) be a Banach space, and consider a bounded linear map \(\varphi\!: V\to X\). Then there exists a linear map \(\varphi'\!:V\to V\)
	such that \(\varphi\simeq\varphi'\).
\end{corollary}
\begin{proof}
Since \(V\) is finite dimensional, \(\varphi\) has finite dimensional image in \(X\), and therefore we may factor it as
\begin{diagram*}
	V \ar[r,"\tilde\varphi"] & \img(\varphi) \ar[r, hook] & X
\end{diagram*}
so that \(\varphi\simeq\tilde\varphi\) since the inclusion \(\img(\varphi)\to X\) is an isometry (when \(\img(\varphi)\) is endowed with the norm induced by \(X\)). Now, since \(\img(\varphi)\) is finite-dimensional of dimension \(\leq\dim(V)\),
we can find some embedding \(\iota\!:\img(\varphi)\to V\) which, by the equivalence of norms in finite-dimensional vector spaces, will satisfy the conditions of Proposition \ref{prop:preceq-composition-relations-1}.
Setting \(\varphi' = \iota\circ\tilde\varphi\) completes the proof.
\end{proof}
\begin{lemma}
	Let \(V\) be a Banach algebra, and let \(\varphi\!:V\to V\) be a bounded linear map. Then \(\varphi\preceq\id\).
\end{lemma}
\begin{proof}
Since \(\varphi\) is continuous, we have for all \(a\in V\) that
\[ \|\varphi(a)\| \leq \|\varphi\|\cdot \| a \| \]
which immediately tells us that \(\varphi\preceq\id\).
\end{proof}

For our purposes, the below Lemma \ref{lemma:preceq-yields-finite-implies-finite} is fundamental: it is what allows us to relate the relation \(\preceq\) to the questions investigated in this document.

\begin{lemma}\label{lemma:right-continuous-monotone-limsup-inequality}
	Let \(f\!:\R\to\R\) be an order-preserving right-continuous function. Then, for any sequence \((x_k)_{k=1}^\infty\), we have
	\[ \limsup_{k\to\infty}f(x_k) \leq f(\limsup_{k\to\infty}x_k). \]
\end{lemma}
\begin{proof}
An order-preserving (i.e.\ monotone) function is right continuous if and only if it preserves infima. Therefore, we have
\[ \limsup_{k\to\infty}f(x_k) = \inf_{k_0\geq 0}\{ \sup_{k\geq k_0}f(x_k) \} \leq \inf_{k_0\geq k}\{ f(\sup_{k\geq k_0}x_k) \} = f(\inf_{k_0\geq k}\{ \sup_{k\geq k_0}x_k \}) = f(\limsup_{k\to\infty}x_k) \]
as desired.
\end{proof}
\begin{lemma}\label{lemma:preceq-yields-finite-implies-finite}
	Let \(V\) be a Banach algebra, and let \(\varphi,\psi\!:V\to V\) be bounded linear maps. Then the following statements are equivalent.
	\begin{enumerate}[label=(\alph*)]
	\item The two maps satisfy \(\varphi\preceq\psi\).
	\item For all \(a\in V\), and any sequence of invertibles \(u_k\in V\), we have
	\[ \limsup_{k\to\infty}\|\psi(u_kau_k^{-1})\| < \infty \implies \limsup_{k\to\infty}\|\varphi(u_kau_k^{-1})\| < \infty. \]
	\end{enumerate}
\end{lemma}
\begin{proof}
To see that (a) implies (b), observe that by Lemma \ref{lemma:right-continuous-monotone-limsup-inequality} we have
\[ \limsup_{k\to\infty}\|\varphi(u_kau_k^{-1})\| \leq \limsup_{k\to\infty}f_a\!\left(\|\psi(u_kau_k^{-1})\|\right) \leq f_a\!\left(\limsup_{k\to\infty}\|\psi(u_kau_k^{-1})\|\right) < \infty \]
which is the statement we wanted.

To show that (b) implies (a), we must, for each \(a\in V\), construct a function \(f_a\!:\R_{\geq 0}\to\R_{\geq 0}\) with some properties, and in particular it must not depend
on the conjugacy class of \(a\). Thus, consider the functions
\[ f_a^\pre(x) := \sup\left\{ \|\varphi(uau^{-1})\| \mid u\in V\text{ invertible, }\|\psi(uau^{-1})\| \leq x \right\} \]
and
\[ f_a(x_0) := \lim_{x\to x_0^+}f_a^\pre(x). \]
We must show that \(f_a^\pre\) is well-defined; in other words, we must show that it is finite. If there is some \(x\in\R_{\geq 0}\) such that \(f_a^\pre(x)=\infty\), then there must
in particular exist some sequence of invertibles \(u_k\in V\) such that \(\lim_{k\to\infty}(\|\varphi(u_kau_k^{-1})\|) = \infty\) and such that, for all \(k\), we have \(\|\psi(u_kau_k^{-1})\|\leq x\).
However, by the latter condition we have
\[ \limsup_{k\to\infty}\|\psi(u_kau_k^{-1})\| < \infty \implies  \limsup_{k\to\infty}\|\varphi(u_kau_k)\| < \infty \]
which contradicts the aformentioned divergence. Therefore, we must have \(f_a^\pre(x) < \infty\) for all \(x\in\R_{\geq 0}\). As an immediate corollary of the definition, we see that \(f_a\)
is right-continuous and order-preserving. Finally, we have that, for all invertible \(u\in V\),
\[ \|\varphi(uau^{-1})\| \leq f_a^\pre(\|\psi(uau^{-1})\|) \leq f_a(\|\psi(uau^{-1})\|) \]
so that \(\varphi\preceq\psi\) as desired.
\end{proof}

In particular, we can now relate pathwise \(\varphi\)-boundedness and pathwise \(\psi\)-boundedness whenever \(\varphi\) and \(\psi\) are related, which is the concrete reason
one might care about the \(\preceq\) relation in our context.
\begin{proposition}\label{prop:hadamard-pathwise-inclusions}
	Let \(V\) be a Banach algebra, let \(z\in V\), and let \(u_k\to z\) be a sequence of invertibles. If \(\varphi\preceq\psi\), then
	\[ \frakS_{\pw}^{\psi}(z;u_k)\subseteq\frakS_{\pw}^\varphi(z;u_k). \]
	In particular,
	\begin{enumerate}[label=(\alph*)]
	\item if \(\varphi\simeq\psi\), then \(\frakS_{\pw}^{\psi}(z;u_k)=\frakS_{\pw}^\varphi(z;u_k)\), and
	\item for all \(\varphi\!:V\to V\), we have \(\frakS_{\pw}(z;u_k) \subseteq\frakS_{\pw}^\varphi(z;u_k)\).
	\end{enumerate}
\end{proposition}
\begin{proof}
Suppose that \(\varphi\preceq\psi\), and that \(a\in\frakS_{\pw}^\psi(z;u_k)\).Then, by Lemma \ref{lemma:preceq-yields-finite-implies-finite},
\[ \lim_{k\to\infty}\|\psi(u_kau_k^{-1})\| < \infty \implies \lim_{k\to\infty}\|u_kau_k^{-1}\| < \infty \]
which means \(a\in\frakS^{\varphi}_{\pw}(z;u_k)\). The remaining assertions follow trivially.
\end{proof}
\begin{corollary}\label{corollary:phi-simeq-psi-implies-same-S}
	Let \(V\) be a Banach algebra, and let \(\varphi\preceq\psi\). Then, for all \(z\in V\),
	\[ \frakS^\psi(z) \subseteq \frakS^\varphi(z)\quad\text{and}\quad \frakS_\cup^\psi(z) \subseteq \frakS_\cup^\varphi(z). \]
	In particular, if \(\varphi\simeq\psi\), then
	\[ \frakS^\psi(z) = \frakS^\varphi(z)\quad\text{and}\quad \frakS_\cup^\psi(z) = \frakS_\cup^\varphi(z). \]
\end{corollary}
\begin{proof}
Apply Proposition \ref{prop:hadamard-pathwise-inclusions} and Proposition \ref{prop:S-and-S_V-as-pathwise-cap-or-cup}.
\end{proof}

\subsection{Existence With a Modifier}\label{subsection:bounded-modified-existence}
We first apply Corollary \ref{corollary:phi-simeq-psi-implies-same-S} to reduce a particular case of computing \(\frakS_\cup^\varphi(-)\) to computing \(\frakS_\cup(-)\), which we have already done. To do this, we need
a preliminary lemma which tells us that a particularly interesting class of \(\varphi\)'s is equivalent to the identity map. The proof makes use of Gershgorin's circle theorem;
see e.g.\ \cite[p.\ 389,\, Corollary 6.1.3]{horn-johnson-matrix-analysis-2nd-edition}.

\begin{theorem}[Gershgorin's circle theorem]
	Let \(A=(a_{ij})\) be a complex \(n\times n\) matrix, let \(R_i = \sum_{j\not=i}|a_{ij}|\), and let \(D(a_{ii},R_i)\) be the disk of radius \(R_i\) centered at \(a_{ii}\). Then
	all eigenvalues are contained within the region \(\bfG := \bigcup_{i}D(a_{ii},R_i)\). Furthermore, each connected component of \(\bfG\) contains exactly as many eigenvalues (counted
	with their algebraic multiplicty) as it contains \(a_{ii}\)'s.
\end{theorem}
\begin{lemma}\label{lemma:J-is-equivalent-to-id}
	Let \(J = \1-I\), and let \(\phi\!:\End(\C^n)\to\End(\C^n)\) be the bounded linear map given by \(\phi\!:A\mapsto J*A\). Then \(\phi\simeq\id\).
\end{lemma}
\begin{proof}
Observe that we trivially have \(\phi\preceq\id\), so it remains only to check that \(\id\preceq\phi\). We first note the following inequalities
\[ \| A \| \leq c\sum_{i,j}|a_{ij}| \leq c'\| J*A \| + c\sum_{i}|a_{ii}|, \]
which we derive from the equivalences of norms on finite-dimensional vector spaces. We will use Gershgorin's circle theorem to estimate the latter sum. Let \(\bfG\) be
as in the statement of Gershgorin's circle theorem, and let \((\lambda_i)\) be the eigenvalues of \(A\), counted with algebraic multiplicity and numbered such that \(\lambda_i\)
is in the same connected component of \(\bfG\) as \(a_{ii}\), which is possible by Gershgorin. We then have
\[ \sum_i |a_{ii}| = \sum_i |a_{ii} - \lambda_i + \lambda_i| \leq \sum_i|a_{ii}-\lambda_i| + \sum_i|\lambda_i|.\]
We now specialize down to estimating just one term \(|a_{ii}-\lambda_i|\). Let \(\bfG_i\) be the connected component of \(\bfG\) containing \(a_{ii}\). Since \(\bfG_i\) is the union of a number of disks \(D(a_{jj},R_j)\),
it is geometrically immediate that any two points of \(\bfG_i\) are certainly within \(2\sum_{a_{jj}\in\bfG_i}R_j\) of each other, and therefore, since \(\lambda_i\in\bfG_i\), we may as well
make an even more extreme estimate summing over \emph{all} \(j\), so
\[ |a_{ii}-\lambda_i| \leq 2\sum_{j=1}^nR_j = 2\sum_j \sum_{k\not=j}|a_{jk}| \leq c''\|J*A\|. \]
Putting this together, we get
\[ \| A \| \leq c'\| J*A \| + c\sum_i c''\|J*A\| + \sum_i|\lambda_i| = (c'+ncc'')\|J*A\| + c\sum_i|\lambda_i| \]
which shows that \(\id\preceq\phi\), as desired.
\end{proof}

\begin{theorem}\label{thm:hadamard-local-infimum-J-theorem}
	Denote by \(\bbJ\) the set of bounded linear maps \(\End(\C^n)\to\End(\C^n)\) that under some basis of \(\C^n\) are of the form \(A\mapsto (\1 - I)*A\), and let \(Z\in\End(\C^n)\). If \(\varphi\in\bbJ\), then
	\[ \frakS_\cup(Z) = \frakS_\cup^\varphi(Z). \]
	In particular, \(\frakS_\cup^\varphi(Z) = \frakS_{\ker}(Z)\).
\end{theorem}
\begin{proof}
Supposing we can prove the first assertion, the second follows by Corollary \ref{corollary:S-vee-equals-S-ker}. To prove our main claim,
apply Corollary \ref{corollary:phi-simeq-psi-implies-same-S} and Lemma \ref{lemma:J-is-equivalent-to-id}.
\end{proof}

Now, in Lemma \ref{lemma:J-is-equivalent-to-id}, we used a rather specialized trick to show that \((J*)\simeq\id\), and then used general theory to compute \(\frakS_\cup^J(-)\). This trick has no reason
whatsoever to work for more general choices of pairs \((\varphi,\id)\), so we have to change tactics completely if we want a hope of understanding the wider situation. Based on the results of
Section \ref{subsection:bounded-unmodified-existence}, we expect that \(\frakS_\cup^\varphi(Z)\) should be related to some modified version of the sets \(\frakS_C(Z)\) for \(C\in\frakC_{-1}(Z)\). We introduce this modification now:

\begin{definition}\label{definition:frakS_c^phi(z)}
	Let \(V\) be a Banach algebra, let \(\varphi\!:V\to X\) be a bounded linear map to a Banach space \(X\), and let \(c,z\in V\) be such that \(zc = cz = 0\). Define the set
	\[ \frakS_{c}^\varphi(z) := \ker(a\mapsto \varphi(zac)) = \{ a\in V\mid \varphi(zac) = 0 \}. \]
	Note that \(\frakS_c(z) = \frakS_c^{\id}(z)\).
\end{definition}
\begin{proposition}\label{prop:hadamard-S_C^H-conjugation}
	Let \(V\) be a Banach algebra, let \(\varphi\!:V\to X\) be a bounded linear map, let \(c,z\in V\), and let \(p\in V^\times\). Then
	\[ \frakS_{c}^\varphi(zp) = p^{-1} \frakS_c^\varphi(z) \quad \text{and} \quad \frakS_{p^{-1}c}^\varphi(z) = \frakS_c^\varphi(z)p. \]
	In particular, \(\frakS_{p^{-1}c}^\varphi(zp) = p^{-1}\frakS_c^\varphi(z)p\).
\end{proposition}
\begin{proof}
These equalities follow immediately from
\[ \varphi(zac) = \varphi((zp)(p^{-1}a)c) \quad \text{and}\quad \varphi(zac) = \varphi(z(ap)(p^{-1}c)) \]
respectively. The final statement is an immediate combination of the first two.
\end{proof}
\begin{proposition}
	Let \(V\) be a Banach algebra, let \(\varphi\!:V\to X\) be a bounded linear map, let \(z\in V\), and let \(c\in\frakC_{-1}(Z)\not=\varnothing\). Then
	\[ \frakS_c^\varphi(z)\subseteq\frakS_\cup^\varphi(z). \]
\end{proposition}
\begin{proof}
Since \(c\in\frakC_{-1}(Z)\), we know that there is at least one corresponding good path \(z+te(t)\) converging to \(z\). Thus, for any \(a\in\frakS_c^\varphi(z)\) we have
\[ \varphi\left( (z+te(t))A(z+te(t))^{-1} \right) = \varphi(zact^{-1} + O(1)) = O(1) \]
where the last equality uses linearity and continuity. As in Lemma \ref{lemma:good-path-annihilates-coefficient-implies-vee}, we are done.
\end{proof}
\begin{corollary}\label{corollary:hadamard-S-vee-contains-union-of-hadamard-SC}
	Let \(V\) be a Banach algebra, let \(\varphi\!:V\to X\) be a bounded linear map, and let \(z\in V\). Then
	\[ \frakS_\cup^\varphi(z) \supseteq \bigcup_{c\in\frakC_{-1}(z)}\frakS_c^\varphi(z). \]
\end{corollary}

\noindent Based on this, one can make a reasonably natural conjecture: that the above inclusion should be an equality.

\begin{conjecture}\label{conjecture:hadamard-general}
	Let \(V\) be a Banach algebra, and let \(\varphi\!:V\to X\) be a bounded linear map. Then, for all \(z\in V\),
	\[ \frakS_\cup^\varphi(z) = \bigcup_{c\in\frakC_{-1}(z)}\frakS_c^\varphi(z). \]
\end{conjecture}
\begin{remark}
	When \(X=\End(\C^n)\) and \(\varphi=\id\), the above is trivially true with the theory we have developed. In particular, we know that in fact
	\[ \frakS_\cup^{\id}(Z) = \frakS_\cup(Z) = \frakS_C(Z) = \frakS_C^{\id}(Z) \]
	and so we also have \(\frakS_{C_1}^{\id}(Z) = \frakS_{C_2}^{\id}(Z)\) for all \(C_1,C_2\in\frakC_{-1}(Z)\). In other words, when \(\varphi=\id\), the conjecture is easily verified and is fairly tautological.
	However, in general, we have no reason to expect \(\frakS_{C_1}^\varphi(Z)\) to be equal to \(\frakS_{C_2}^\varphi(Z)\), but nonetheless we know that for all \(C\in\frakC_{-1}(Z)\), we have \(\frakS_C^\varphi(Z)\subseteq\frakS_\cup^\varphi(Z)\)
	following the exact same proof as in the case \(\varphi=\id\) (see below).
\end{remark}

In the cases we are most interested in, Conjecture \ref{conjecture:hadamard-general} turns out to be true.

\begin{lemma}\label{lemma:inverse-of-block-map}
	Let \(X\) be a vector space, and suppose there is a direct sum decomposition of \(X\) of the form
	\[ X = X_1 \oplus X_2. \]
	Let \(U\!:X\to X\) be a linear map of the form
	\[ U = \begin{pmatrix} Q & R \\ S & T \end{pmatrix} \]
	such that \(Q\!:X_1\to X_1\) is invertible. Then
	\begin{enumerate}[label=(\alph*)]
	\item the map \(P := T - S Q^{-1} R\) is invertible if and only if \(U\) is invertible, and
	\item if \(U\) is invertible, the inverse is given by
	\[ U^{-1} = \begin{pmatrix}
        Q^{-1} + Q^{-1} R P^{-1}SQ^{-1} & -Q^{-1}R P^{-1} \\ 
        -P^{-1}SQ^{-1} & P^{-1} \end{pmatrix} \]
	\end{enumerate}
\end{lemma}
\begin{proof}
Note that we can write
\[ U = \begin{pmatrix} Q & R \\ S & T \end{pmatrix} = \begin{pmatrix} \id_{X_1} & 0 \\ SQ^{-1} & \id_{X_2} \end{pmatrix} \begin{pmatrix} Q & 0 \\ 0 & P \end{pmatrix} \begin{pmatrix} \id_{X_1} & Q^{-1}R \\ 0 & \id_{X_2} \end{pmatrix} \]
and that the two surrounding maps are invertible on account of being lower/upper block triangular. Therefore, the invertibility of \(U\) is equivalent to the invertibility of the middle map,
whose invertibility, due to being block diagonal, depends only on \(Q\) and \(P\). Since \(Q\) is invertible by assumption, we see that \(P\) is invertible if and only if \(U\) is. This proves (a).
The proof of (b), then by (a) we know that \(P\) is invertible, and that \(U^{-1}\) is of the given form is just a standard computation coming from the expression given above.
\end{proof}
\begin{theorem}\label{thm:hilbert-hadamard-general}
	Let \(X, Y\) be Hilbert spaces, let \(\varphi\!:\sfB(X)\to Y\) be a bounded linear map, and let \(Z'\in\sfB(X)\) be of generalized index zero. Then
	\[ \frakS_\cup^\varphi(Z') = \bigcup_{C\in\frakC_{-1}(Z)}\frakS_C^\varphi(Z'). \]
\end{theorem}
\begin{proof}
Using Theorem \ref{thm:polar-decomposition}, write \(Z' = MZ\) where \(M\) is invertible and \(Z\) is self-adjoint. Now, consider the decomposition
\[ X = \img(Z) \oplus \ker(Z), \quad \text{under which } Z = \begin{pmatrix} Z_{11} & 0 \\ 0 & 0 \end{pmatrix} \]
where \(Z_{11}\) is invertible. From now on, we will work in the setting of this decomposition.

If \(A\in\frakS_\cup^\varphi(Z')\), then there is some sequence \(U_n'\to Z'\) exhibiting this fact. Since \(Z' = MZ\), we may write \(U_n' = MU_n\) where \(U_n\to Z\), i.e.
\[ U_n = \begin{pmatrix} Z_{11} + Q_n & R_n \\ S_n & T_n \end{pmatrix}, \quad \text{where } Q_n,R_n,S_n,T_n\to 0 \text{ as }n\to\infty. \]
For simplicity, we write \(\tilde{Q}_n = Z_{11} + Q_n\). In what follows, we will neglect to include the subscript \(n\) to keep the notation uncluttered.

We are interested in the product \(U'AU'^{-1}\), which by Lemma \ref{lemma:inverse-of-block-map} can be written as
\[ MUAU^{-1}M^{-1} = M \begin{pmatrix} \tilde{Q} & R \\ S & T \end{pmatrix} \begin{pmatrix} A_{11} & A_{12} \\ A_{21} & A_{22} \end{pmatrix}
\begin{pmatrix} \tilde{Q}^{-1} + \tilde{Q}^{-1} RP^{-1}S\tilde{Q}^{-1} & -\tilde{Q}^{-1}RP^{-1} \\ -P^{-1}S\tilde{Q}^{-1} & P^{-1} \end{pmatrix} M^{-1} \]
where \(P_n = T_n - S_n(I + Q_n)^{-1}R_n\). We consider the top-right block of this, which (as in the proof of Theorem \ref{thm:hilbert-hadamard-general}) is given by
\[ (\tilde{Q}_nA_{12} + R_nA_{22} - (\tilde{Q}_nA_{11} + R_nA_{21})\tilde{Q}_n^{-1}R_n)P_n^{-1}. \]
Using that \(\tilde{Q}_n = Z_{11} + Q_n\), we note that the above may be rewritten to be of the form
\[ Z_{11}A_{12}P^{-1}_n + B_nP^{-1}_n \]
where \(B_n \to 0\). Now, one easily checks that in all blocks, any unbounded term has its unboundedness come from \(P^{-1}_n\), and furthermore,
that the top-right is the only block containing a term of the form \(\text{const}\cdot P^{-1}_n\). As a result, after applying the reverse triangle inequality, one sees that in
\begin{align*}
	\| \varphi(M U_n A U_n^{-1} M^{-1}) \| &= \left\| \varphi\!\left(M \begin{pmatrix} \bullet & Z_{11}A_{12}P_n^{-1} + B_nP^{-1}_n \\ \bullet & \bullet \end{pmatrix} M^{-1}\right) \right\| \\
	&\geq \left| \left\| \varphi\!\left(M \begin{pmatrix} \bullet & B_nP_n^{-1} \\ \bullet & \bullet \end{pmatrix} M^{-1}\right) \right\|
	- \left\| \varphi\!\left(M \begin{pmatrix} 0 & Z_{11}A_{12}P^{-1}_n \\ 0 & 0 \end{pmatrix} M^{-1}\right) \right\| \right|
\end{align*}
the term
\[ \left\| \varphi\!\left(M \begin{pmatrix} 0 & Z_{11}A_{12}P^{-1}_n \\ 0 & 0 \end{pmatrix} M^{-1}\right) \right\| \]
dominates as \(n\to\infty\). By assumption, \(A\) and the sequence \(U_n\) were chosen such that this is finite, and therefore, we deduce that we must necessarily have
\[ \varphi\!\left(M \begin{pmatrix} 0 & Z_{11}A_{12}P^{-1}_n \\ 0 & 0 \end{pmatrix} M^{-1}\right) = 0 \]
once \(n\) is sufficiently large. Clearly, we have that
\[ C := \begin{pmatrix} 0 & 0 \\ 0 & P_n^{-1} \end{pmatrix} \in \frakC_{-1}(Z) \quad \leadsto \quad C' := CM^{-1}\in\frakC_{-1}(Z') \]
so that
\begin{align*}
	0 = \varphi\!\left(M \begin{pmatrix} 0 & Z_{11}A_{12}P^{-1}_n \\ 0 & 0 \end{pmatrix} M^{-1}\right) = \varphi(M Z A C M^{-1}) = \varphi(Z'AC')
\end{align*}
We conclude that
\[ \frakS_\cup^\varphi(Z')\subseteq \bigcup_{C\in\frakC_{-1}(Z')}\frakS_C^\varphi(Z') \]
as desired.
\end{proof}

\begin{theorem}\label{thm:hilbert-hadamard-general-dual}
	Let \(X\) be a Hilbert space, and let \(Z\in\sfB(X)\). We let
	\[ \frakN_\cup^{\varphi,*}(A,Z) = \limsup_{r\to 0}\inf_{\| U - Z \|<r}\| \varphi(U^{-1}AU) \|, \]
	and define the sets
	\[ \frakS_\cup^{\varphi,*}(Z) = \{ A \in\sfB(X) \mid \frakN_\cup^{\varphi,*}(A,Z) < \infty \}, \quad \frakS_C^{\varphi,*}(Z) = \{ A\in\sfB(X) \mid \varphi(CAZ) = 0 \}. \]
	Then, if \(Z\) is of generalized index zero,
	\[ \frakS_\cup^{\varphi,*}(Z) = \bigcup_{C\in\frakC_{-1}(Z)} \frakS_C^{\varphi,*}(Z). \]
\end{theorem}
\begin{proof}
First, note that one may reduce to the case where \(Z\) is self-adjoint. In particular, for any \(Z'\in\sfB(X)\), we write \(Z' = MZ\) with \(M\) invertible
and \(Z\) self-adjoint by Theorem \ref{thm:polar-decomposition}. Then dual versions of the tranformation laws for \(\frakS_\cup^\varphi(-)\) and \(\frakS_C^\varphi(-)\)
mean that it will suffice to consider \(\frakS_\cup^{\varphi,*}(Z)\). After this, simply apply a dual version of the proof strategy of Theorem \ref{thm:hilbert-hadamard-general},
considering the bottom-left block instead of the top-right block.
\end{proof}

While we have already solved the case when \(X\) is finite-dimensional and \(\varphi(A)=J*A\) by reducing it to the unmodified version of the problem, which in turn was done by a clever trick, we may employ
Theorem \ref{thm:hilbert-hadamard-general} to give a different proof of Theorem \ref{thm:hadamard-local-infimum-J-theorem}. To begin, we make the following observations:

\begin{lemma}\label{lemma:CZ=0-implies-ZAC-nilpotent}
	Let \(Z,C\in\End(\C^n)\) be such that \(CZ = 0\). Then for any matrix \(A\) we have \((ZAC)^2 = 0\), and in particular, \(ZAC\) is nilpotent.
\end{lemma}
\begin{proof}
Trivially, \((ZAC)^2 = ZACZAC = ZA(CZ)AC = 0\).
\end{proof}

\begin{proposition}\label{prop:J-hadamard-Ann-is-Ann}
	Let \(J = \1 - I\), let \(\varphi\!:\End(\C^n)\to\End(\C^n)\) be the map given by \(A\mapsto J*A\), and let \(Z,C\in\End(\C^n)\) be matrices such that \(CZ = 0\). Then
	\[ \frakS_C^\varphi(Z) = \frakS_C(Z). \]
\end{proposition}
\begin{proof}
If \(ZAC = 0\), then clearly also \(J*(ZAC)=0\). Conversely, \(J*(ZAC)=0\) is equivalent to saying that \(ZAC\) is diagonal, which means
in particular that any basis vector \(e_k\) is an eigenvector of \(ZAC\). However, by Lemma \ref{lemma:CZ=0-implies-ZAC-nilpotent} the matrix
\(ZAC\) is nilpotent, so all eigenvalues are zero. Therefore, \(ZACe_k = 0\), so that \(ZAC=0\).
\end{proof}

\begin{proof}[Alternative proof of Theorem \ref{thm:hadamard-local-infimum-J-theorem}]
Let \(Z\in\End(\C^n)\) and let \(J = \1 - I\). By Proposition \ref{prop:J-hadamard-Ann-is-Ann} we have that for all \(C,C'\in\frakC_{-1}(Z)\),
\[ \frakS_C^{(J*)}(Z) = \frakS_C(Z) = \frakS_{\cup}(Z) = \frakS_{C'}(Z). \]
Therefore, the union in Theorem \ref{thm:hilbert-hadamard-general} collapses to give us
\[ \frakS_\cup^{(J*)}(Z) = \bigcup_{C\in\frakC_{-1}}\frakS_C(Z) = \frakS_\cup(Z) \]
as desired.
\end{proof}

With having now proven Conjecture \ref{conjecture:hadamard-general} in an interesting case and given an application of it, we will point out a downside of it: in practice, one must in principle check an
infinite number of conditions (one for each \(C\in\frakC_{-1}(Z)\)) in order to know if some \(A\) is in \(\frakS_\cup^\varphi(Z)\) or not.
It would be extremely convenient if one could somehow reduce this to checking a \emph{single} condition. However, as it turns out, this is too good to be true in general.

\begin{example}\label{example:infinite-union-necessary}
	Here is an example showing that for at least one choice of \(H,Z\), there is no \(C\in\frakC_{-1}(Z)\) for which \(\frakS_\cup^{(H*)}(Z) = \frakS_C^{(H*)}(Z)\). Let
	\[ Z = \begin{pmatrix} 1 & 0 & 0 \\ 0 & 0 & 0 \\ 0 & 0 & 0 \end{pmatrix},\quad H = \begin{pmatrix} 0 & 0 & 1 \\ 0 & 0 & 0 \\ 0 & 0 & 0 \end{pmatrix}. \]
	In this case, by Theorem \ref{thm:hilbert-space-gen-index-zero-C1=C'}, we know that all \(C\in\frakC_{-1}(Z)\) are of the form
	\[ C = \begin{pmatrix} 0 & 0 & 0 \\ 0 & x & y \\ 0 & z & w \end{pmatrix}. \]
	Now, a computation shows that
	\[ ZAC = \begin{pmatrix} 0 & xa_{12} + za_{13} & ya_{12} + wa_{13} \\ 0 & 0 & 0 \\ 0 & 0 & 0 \end{pmatrix}, \quad\text{where } A=(a_{ij}). \]
	Now the condition that \(H*(ZAC) = 0\) just becomes \(ya_{12} + wa_{13} = 0\).

	For any pair \((y,w)\in\C^2\backslash\{(0,0)\}\), one may find a matrix \(C\in\frakC_{-1}(Z)\) of the above form. Indeed, if they are non-zero, then choosing \(z=0\) and \(x=1\)
	yields that the bottom right of \(C\) is an upper triangular matrix, hence invertible. If only one is zero, then one can perform a similar trick. Now, by the above computation and
	Theorem \ref{thm:hilbert-hadamard-general}, we have
	\[ \frakS_\cup^{(H*)}(Z) = \bigcup_{C\in\frakC_{-1}(Z)}\frakS_C^{(H*)}(Z) = \bigcup_{(y,w)\in\C^2\backslash\{(0,0)\}}\{ (a_{ij}) \mid ya_{12} + wa_{13} = 0 \}. \]
	However, since for any pair \((a_{12},a_{13})\in\C^2\) we may find some pair \((y,w)\not=(0,0)\) such that \(ya_{12} + wa_{13} = 0\), we see that
	\[ \frakS_\cup^{(H*)}(Z) = \End(\C^3). \]
	On the other hand, it is clear that for all \(C\in\frakC_{-1}(Z)\), we have a \emph{strict} inclusion
	\[\frakS_C^{(H*)}(Z)\subsetneq\End(\C^n).\]
\end{example}

\begin{remark}
	The above example also implies that, in general, it is impossible to pick a finite subset of \(\frakC_{-1}(Z)\) which completely represents \(\frakS_\cup^H(Z)\). If it were possible,
	then it is easy to see that in fact \emph{one} \(C\in\frakC_{-1}(Z)\) would be enough, which is ruled out by the example.
\end{remark}

\subsection{Modifiers Equivalent to the Identity}\label{subsection:modifiers-equivalent-to-the-identity}
We finish the section with a characterizing condition for those \(\varphi\) such that \(\frakS_\cup^\varphi(-) = \frakS_\cup(-)\). This yields a necessary
criterion for having \(\varphi\simeq\id\), and a complete classification in the case when \(\varphi\) is given by a Hadamard product.

\begin{lemma}\label{lemma:hadamard-S_C-inclusion}
	Let \(V\) be a Banach algebra, \(X\) a Banach space, \(\varphi\!: V\to X\) be a bounded linear map, and let \(c,z\in V\) be such that \(cz = zc = 0\). Then \(\frakS_c(z)\subseteq\frakS_c^\varphi(z)\).
\end{lemma}
\begin{proof}
If \(a\in\frakS_c(z)\), we will have \( zac=0 \), so that \(\varphi(zac)=0\), and hence \(a\in\frakS_c^\varphi(z)\).
\end{proof}
\begin{proposition}\label{prop:S_V^phi=S_V-iff-S_C^phi=S_C}
	Let \(X\) be a Hilbert space, \(Y\) a Banach space, \(\varphi\!:\sfB(X)\to Y\) be a bounded linear map, and let \(Z\in\sfB(X)\) be of generalized index zero. Then
	\[ \frakS_\cup^\varphi(Z) = \frakS_\cup(Z) \iff \forall C\in\frakC_{-1}(Z),\, \frakS_C^\varphi(Z) = \frakS_C(Z). \]
\end{proposition}
\begin{proof}
If the latter condition is satisfied, then by Theorem \ref{thm:hilbert-hadamard-general} we will have
\[ \frakS_\cup^\varphi(Z) = \bigcup_{C\in\frakC_{-1}(Z)}\frakS_{C}^\varphi(Z) = \bigcup_{C\in\frakC_{-1}(Z)}\frakS_C(Z) = \frakS_\cup(Z). \]
Conversely, suppose \(\frakS_\cup^\varphi(Z)=\frakS_\cup(Z)\) and fix some \(C\in\frakC_{-1}(Z)\). Then, by Corollary \ref{corollary:hilbert-space-S_V=S_ker} and Lemma \ref{lemma:hadamard-S_C-inclusion},
\[ \frakS_C(Z)\subseteq\frakS_C^\varphi(Z)\subseteq\frakS_\cup^\varphi(Z) = \frakS_\cup(Z) = \frakS_C(Z) \]
so that \(\frakS_C^\varphi(Z) = \frakS_C(Z)\) as desired.
\end{proof}

\begin{definition}
	An \emph{orthogonal algebra} is an algebra \(\calA\) with the property that \(xy=0\) for all \(x,y\in\calA\). Given an orthogonal subalgebra \(\calT\subseteq\sfB(X)\) where \(X\) is a Hilbert space, we define two associated
	closed subspaces
	\[ \img(\calT) := \bigcap_{S\in\calT}\ker(S),\quad \ker(\calT) := \img(\calT)^\perp. \]
\end{definition}
\begin{proposition}\label{prop:orthogonal-subalgebra-of-B(X)-decomposition}
	Let \(X\) be a Hilbert space, consider an orthogonal subalgebra \(\calT\subseteq\sfB(X)\), and let \(T\in\calT\). Then, under the decomposition \(X = \ker(\calT)\oplus\img(\calT)\), we can write \(T\) in the form
	\begin{diagram*}[ampersand replacement=\&]
		\img(\calT)\oplus\ker(\calT) \arrow{r}{\begin{pmatrix} 0 & T_{12} \\ 0 & 0 \end{pmatrix}} \& \img(\calT)\oplus\ker(\calT).
	\end{diagram*}
	In particular, there is a bounded operator \(Z\in\sfB(X)\) of generalized index zero and some \(C\in\frakC_{-1}(Z)\) such that for all \(T\in\calT\) we have \(T = ZTC\).
\end{proposition}
\begin{proof}
By definition of $\img(\calT)$, we have $\img(\calT)\subseteq \ker(T)$. Since \(\calT\) is an orthogonal algebra, we have that \(ST = 0\) for all \(S\in\calT\), so \(\img(T)\subseteq \img(\calT)\). Therefore,
in total, we have
\[ \img(T)\subseteq \img(\calT) \subseteq\ker(T). \]
From this, it is clear that \(T\) has the claimed form. Setting
\[ Z = \begin{pmatrix} \id_{\img(\calT)} & 0 \\ 0 & 0 \end{pmatrix},\quad C = \begin{pmatrix} 0 & 0 \\ 0 & \id_{\ker(\calT)} \end{pmatrix} \]
one obtains the final statement. That \(Z\) is of generalized index zero is clear by Theorem \ref{thm:generalized-index-zero-iff-C'-non-empty}, since \(C\in\frakC'(Z)\).
\end{proof}

\begin{lemma}\label{lemma:2nd-order-nilpotents-are-ZAC}
	Let \(X\) be a Hilbert space, and let \(T\in\sfB(X)\). Then \(T^2 = 0\) if and only if there exists \(Z\in\sfB(X)\) of generalized index zero and \(C\in\frakC_{-1}(Z)\) such
	that \(T = ZAC\).
\end{lemma}
\begin{proof}
If \(T\) is of the form \(T=ZAC\), then trivially \(T^2 = ZACZAC = 0\). Conversely, apply Proposition \ref{prop:orthogonal-subalgebra-of-B(X)-decomposition} to the subalgebra \(\C[T]\subseteq\sfB(X)\)
generated by \(T\), which is an orthogonal algebra since \(T^2=0\).
\end{proof}

\begin{theorem}\label{thm:S_V^phi=S_V-implies-phi(2nd-ord-nilp)-nonzero}
	Let \(X,Y\) be Hilbert spaces, and let \(\varphi\!:\sfB(X)\to Y\) be a bounded linear map. Then the following statements are equivalent:
	\begin{enumerate}[label=(\roman*)]
	\item For all \(Z\in\sfB(X)\) of generalized index zero, we have \(\frakS_\cup^\varphi(Z) = \frakS_\cup(Z)\).
	\item For all \(Z\in\sfB(X)\) of generalized index zero and for all \(C\in\frakC_{-1}(Z)\), we have \(\frakS_C^\varphi(Z) = \frakS_C(Z)\).
	\item For all \(T\in\sfB(X)\) such that \(T^2=0\), we have \(\varphi(T)=0\) if and only if \(T=0\).
	\end{enumerate}
\end{theorem}
\begin{proof}
The equivalence between (i) and (ii) follows by Prooposition \ref{prop:S_V^phi=S_V-iff-S_C^phi=S_C}. We prove that (iii) is equivalent to (ii). Suppose that (ii) holds; it is equivalent to saying that
\[ \forall Z\in\sfB(X),\,C\in\frakC_{-1}(Z),\quad [\varphi(ZAC) \not= 0 \iff ZAC \not= 0] \]
and thus, if \(T\not=0\) is such that \(T^2 = 0\), Lemma \ref{lemma:2nd-order-nilpotents-are-ZAC} tells us precisely that \(\varphi(T)\not=0\). Therefore, (ii) implies (iii)

Conversely, if (iii) holds, then this is the same as saying that for all \(T\) such that \(T^2=0\), we have \(\varphi(T)=0\) if and only if \(T=0\).
However, since for all \(A,Z\in\sfB(X)\) where \(Z\) is of generalized index zero and \(C\in\frakC_{-1}(Z)\), we have \((ZAC)^2=0\), this means
\[ ZAC = 0\iff \varphi(ZAC)=0 \]
so that \(\frakS_C^\varphi(Z)=\frakS_C(Z)\). Thus, (iii) implies (ii).
\end{proof}
\begin{corollary}\label{corollary:hadamard-phis-equiv-to-id}
	Let \(\varphi\!:\End(\C^n)\to\End(\C^n)\).
	\begin{enumerate}[label=(\alph*)]
	\item If \(\varphi\simeq\id\) and \(X^2=0\), then \(\varphi(X)=0\) if and only if \(X=0\).
	\item If \(\varphi\) is, under some basis of \(\C^n\), given by \(A\mapsto H*A\) for some matrix \(H\), then
	\[ \frakS_\cup^\varphi(-) = \frakS_\cup(-) \iff \text{all non-diagonal entries of \(H\) are non-zero}\iff \varphi\simeq\id. \]
\end{enumerate}
\end{corollary}
\begin{proof}
To prove (a), note that \(\varphi\simeq\id\) implies that \(\frakS_\cup^\varphi(-)=\frakS_\cup(-)\), and hence Theorem \ref{thm:S_V^phi=S_V-implies-phi(2nd-ord-nilp)-nonzero} implies that \(\varphi(X)\not=0\) unless \(X=0\).
To prove (b), first observe that the implication
\[ \frakS_\cup^\varphi(-) = \frakS_\cup(-) \implies \text{all non-diagonal entries of \(H\) are non-zero} \]
follows from Theorem \ref{thm:S_V^phi=S_V-implies-phi(2nd-ord-nilp)-nonzero}, noting that the basis matrices \(E_{ij}\) satisfy \(E_{ij}^2 = 0\) as long as \(i\not=j\). Next, note that if we let \(\tilde{H}\) be the matrix given by replacing
all non-zero entries in \(H\) by \(1\) and set \(\tilde{\varphi}(A) := \tilde{H}*A\), it is clear that \(\varphi\simeq\tilde\varphi\) by the equivalence of norms on finite-dimensional spaces (for example, one may apply
Proposition \ref{prop:preceq-composition-relations-1}). Furthermore, it is clear that \(J\preceq\tilde{\varphi}\) whenever all non-diagonal entries of \(H\) are non-zero, so by Lemma \ref{lemma:J-is-equivalent-to-id} we get
\[ \text{all non-diagonal entries of \(H\) are non-zero}\implies \varphi\simeq\id. \]
Finally, that
\[ \varphi\simeq\id\implies\frakS_\cup^\varphi(-) = \frakS_\cup(-) \]
follows by Corollary \ref{corollary:phi-simeq-psi-implies-same-S}. This completes the proof.
\end{proof}
\begin{remark}
	In fact, in the above, the equivalence
	\[ \frakS_\cup^\varphi(-) = \frakS_\cup(-) \iff \text{all non-diagonal entries of \(H\) are non-zero} \]
	follows also from noting that diagonal nilpotents are zero, and therefore if \(H\) satisfies the condition on the right, \(X^2=0\), and \(H*X = 0\), then \(X=0\). This is the same proof as
	in Proposition \ref{prop:J-hadamard-Ann-is-Ann}.
\end{remark}
\begin{remark}\label{remark:proof-of-nilpotent-thing}
	The proof of Theorem \ref{thm:S_V^phi=S_V-implies-phi(2nd-ord-nilp)-nonzero} can be replicated more generally whenever one has an analogue of Lemma \ref{lemma:2nd-order-nilpotents-are-ZAC} and one
	knows that Conjecture \ref{conjecture:hadamard-general} is true. In more general settings, however, one may have to contend with \(\frakC_{-1}(-)\) and \(\frakC'(-)\) no longer
	obviously being identical.
\end{remark}

\clearpage
\section{Existence of Unbounded Paths}\label{section:existence-of-unbounded-paths}
This section is dedicated to understanding \(\frakS_\cap(-)\). Before we get on with our main results, though, we point out a few central devices we will use in our strategy.

\subsection{Generalities}\label{subsection:unbounded-generalities}
We start with a remarkably powerful observation based on the following fact: if we have some family of sets \(S_{r}\subseteq\R\), \(r > 0\), such that \(S_{t}\subseteq S_{t'}\) whenever \(t < t'\), then
\[ \lim_{r\to 0}\sup{S_r} < \infty \implies \exists r'>0\text{ such that } \sup{S_{r'}} < 0. \]
Applying this fact to \(\frakS_\cap^\varphi(-)\) yields the following theorem:

\begin{theorem}\label{thm:S(Z)-contained-in-cup-cap-ker}
	Let \(V\) be a Banach algebra, let \(Y\) be a Banach space, let \(\varphi\!:V\to Y\) be a bounded linear map, and let \(z\in V\). Then
	\[ \frakS_\cap^\varphi(z) = \bigcup \left\{ \bigcap_{z'\in U}\frakS_\cap^\varphi(z') : U\subseteq V\text{ open and bounded, }z\in U \right\}. \]
	Here, by \emph{bounded,} we mean \emph{bounded in diameter} (as a subset of a metric space).
\end{theorem}
\begin{proof}
If \(a\in\frakS_\cap^\varphi(z)\), then, by definition,
\[ \lim_{r\to 0}\sup_{\substack{\|u - z\| < r \\ u\in V^\times}}\|\varphi(uau^{-1})\| < \infty. \]
However, for this to hold, it must be that for some \(r' > 0\),
\[ \sup_{\substack{\|u - z\| < r' \\ u\in V^\times}}\|\varphi(uau)^{-1}\| < \infty. \]
In other words,
\begin{align*}
	\frakS_\cap^\varphi(z) &\subseteq \bigcup_{\substack{U\subseteq V \\ U\text{ bounded} \\ z\in U}}\{ a\in V \mid \sup_{u\in U\cap V^\times}\|\varphi(uau)^{-1}\| < \infty \} \\
	&\subseteq \bigcup_{\substack{U\subseteq V \\ U\text{ bounded} \\ z\in U}}\bigcap_{z'\in U} \frakS_\cap^\varphi(z') \subseteq \frakS_\cap^\varphi(z),
\end{align*}
as was to be proven.
\end{proof}

Here is the really useful way we may utilize the above:
\begin{corollary}\label{corollary:S^phi-local-closure-properties}
	Let \(V\) be a Banach algebra, \(Y\) a Banach space, \(\varphi\!:V\to Y\) a bounded linear map, and \(z\in V\). If \(a\in\frakS_\cap^\varphi(z)\), then there is some \(r > 0\), depending on \(a\), such
	that for all \(x\in V\) satisfying \(\|x\| < r\), we have \((1+x)^{-1}a(1+x)\in\frakS_\cap^\varphi(z)\). In particular, the following hold:
	\begin{enumerate}[label=(\roman*)]
	\item For any \(x\in V\), let \(e^x := \sum_{k=0}^\infty \frac{x^k}{k!} \). For all \(a\in\frakS_\cap^\varphi(z)\), there is some \(r' > 0\) such that if \(\|x\| < r'\), then \(e^{-x}ae^x\in\frakS_\cap^\varphi(z)\).
	\item For any \(x\in V\) such that \(e^{\varepsilon x} = f(\varepsilon) + g(\varepsilon)x\) where \(f\) is even and \(g\) is odd, for example when
		\begin{enumerate}[label=(\alph*)]
		\item \(x^2 = 0\), or
		\item \(x^2 = 1\),
		\end{enumerate}
	and for all \(a\in\frakS_\cap^\varphi(z)\), we have \(xax\in\frakS_\cap^\varphi(z)\) and \([x,a]\in\frakS_\cap^\varphi(z)\).
	\end{enumerate}
\end{corollary}
\begin{proof}
Let \(a\in\frakS_\cap^\varphi(z)\). Then, for some \(r_0 > 0\) we have that \(a\in\frakS_\cap^\varphi(z')\) for all \(\|z-z'\| < r_0\) by Theorem \ref{thm:S(Z)-contained-in-cup-cap-ker}. In particular,
if we pick \(r > 0\) so that whenever \(x\in V\) satisfies \(\|x\| < r\) it is small enough that \(1+x\) is invertible and \(\|z - z(1+x)^{-1}\| < r_0\), then
\[ a \in \frakS_\cap(z(1+x)^{-1}) = (1+x)\frakS_\cap(z)(1+x)^{-1}, \]
by Corollary \ref{corollary:hadamard-conjugation-law}, so that \((1+x)^{-1}a(1+x)\in\frakS_\cap(z)\).

To prove (i), note that if \(\|x\|\) is chosen to be small enough (less than some threshold \(r' > 0\)), then
\[ \left\|x + \frac{x^2}{2!} + \frac{x^3}{3!} + \cdots\right\| < r, \]
so that we may employ the first part to see that \(e^{-x}ae^x\in\frakS_\cap^\varphi(z)\).

Let \(\varepsilon\) be such that \(\|\varepsilon x\| < r'\). To prove (ii), note that if \(e^{\varepsilon x} = f(\varepsilon) + g(\varepsilon)x\) with \(f\) even and \(g\) odd, then \(e^{-\varepsilon x} = f(\varepsilon) - g(\varepsilon)x\), so
\begin{align*}
	e^{-\varepsilon x}ae^{\varepsilon x} &= (f(\varepsilon) - g(\varepsilon)x)a(f(\varepsilon) + g(\varepsilon)x)\\
	&= f(\varepsilon)^2a + f(\varepsilon)g(\varepsilon)ax - f(\varepsilon)g(\varepsilon)xa - g(\varepsilon)^2xax \\
	&= f(\varepsilon)^2a + f(\varepsilon)g(\varepsilon)[a,x] - g(\varepsilon)^2xax \in \frakS_\cap^{\varphi}(z)
\end{align*}
which, since \(\frakS_\cap^{\varphi}(z)\) is a vector space, implies that
\[ e^{-\varepsilon x}ae^{\varepsilon x} - f(\varepsilon)^2a = f(\varepsilon)g(\varepsilon)ax - f(\varepsilon)g(\varepsilon)xa - g(\varepsilon)^2xax \in \frakS_\cap^{\varphi}(z). \]
Swapping \(x\) for \(-x\) and noting that \([a,x] = -[x,a]\), we see that, in total,
\begin{align*}
	f(\varepsilon)g(\varepsilon)[a,x] - g(\varepsilon)^2xax &\in \frakS_\cap^{\varphi}(z),\text{ and} \\
	-f(\varepsilon)g(\varepsilon)[a,x] - g(\varepsilon)^2xax &\in \frakS_\cap^{\varphi}(z).
\end{align*}
Once again using that \(\frakS_\cap^\varphi(z)\) is a vector space, we may add and subtract to see that
\[ -2g(\varepsilon)^2xax \in \frakS_\cap^{\varphi}(z),\quad \text{and}\quad 2f(\varepsilon)g(\varepsilon)[a,x]\in\frakS_\cap^\varphi(z). \]
Scaling then proves the result.

We need to show that the provided examples for \(x\) in (ii) satisfy the property we described. This is a reasonably easy computation obtained from just expanding
out the definition of \(e^{\varepsilon x}\).
\begin{enumerate}[label=(\alph*)]
\item If \(x^2=0\), then
\[ e^{\varepsilon x} = 1 + \varepsilon x + \frac{\varepsilon^2 x^2}{2!} + \frac{\varepsilon^3 x^3}{3!} + \cdots = 1+\varepsilon x.  \]
\item If \(x^2=1\), then
\[ e^{\varepsilon x} = 1 + \varepsilon x + \frac{\varepsilon^2}{2!} + \frac{\varepsilon^3 x}{3!} + \cdots = \cosh(\varepsilon) + \sinh(\varepsilon)x. \]
\end{enumerate}
Thus, we are done.
\end{proof}

\subsection{Classification for Existence With a Modifier in Finite Dimensions}\label{subsection:unbounded-modified-existence}
We begin with a construction which will allow us to show there are only two possible ways \(\frakS_\cap^\varphi(Z)\) can look for \(Z\in\End(\C^n)\).
\begin{definition}
	Let \(X\) be a Hilbert space, and \(x,y\in X\). Define the operator
	\[ E_{xy}\!:X\to \Span\{x\}\subseteq X,\quad E_{xy}u := (y,u)x. \]
	In other words,
	\[ E_{xy}y := x,\quad \forall z\in\Span\{y\}^\perp,\, Ez := 0. \]
\end{definition}
\begin{proposition}\label{prop:Exy-properties}
	Let \(X\) be a Hilbert space, and \(x,y,z\in X\) have unit norm.
	\begin{enumerate}[label=(\roman*)]
	\item We have \(E_{xy}E_{yz} = E_{xz}\). In particular, \(E_{xx}^2 = E_{xx}\).
	\item We have \((x,z)=0\) if and only if \(E_{yz}E_{xy} = 0\). In particular, we have \(E_{xy}^2=0\) when \((x,y)=0\).
	\item For any operator \(A\in\sfB(X)\), there is some \(a_{yx}\in\C\) such that
	\[ E_{xy}AE_{xy} = a_{yx}E_{xy}. \]
	\item We have \(a_{yx} = 0\) if and only if \((Ax,y) = 0\). Furthermore,
	\begin{enumerate}[label=(\alph*)]
	\item \(a_{yx} = 0\) for all \(x,y\in X\), if and only if \(A = 0\), and
	\item \(a_{yx} = 0\) for all perpendicular \(x,y\), if and only if \(A = \lambda I\).
	\end{enumerate}
	\end{enumerate}
\end{proposition}
\begin{proof}
To prove (i), write \(u = \mu z + z'\), \(z'\in\Span\{z\}^\perp\), and compute
\[ E_{xy}E_{yz}u = E_{xy}E_{yz}( \mu z + z' ) = \mu E_{xy} y = \mu x = E_{xz}u. \]
To prove (ii), assume \((x,z)=0\), and write \(u = \mu y + y'\), \(y'\in\Span\{y\}^\perp\), and note that
\[ E_{yz}E_{xy} u = E_{yz}E_{xy}(\mu y + y') = \mu E_{yz} x = 0 \]
since \(x\in\Span\{z\}^\perp\). The other direction is similar.

For (iii), write \(u = \mu y + y'\) and \(Ax = a_{yx} y + y''\), where \(y',y''\in\Span\{y\}^\perp\), and note that
\[ E_{xy}AE_{xy}u = \mu E_{xy}Ax = \mu a_{yx} x = a_{yx} E_{xy}u. \]
Finally, for (iv), note that \(a_{yx}=0\) if and only if \(Ax\in\Span\{y\}^\perp\) by the above. If this holds for all \(x,y\), then \((Ax,y) = 0\) for all \(y\), so \(Ax=0\), proving (iv)(a).
For (iv)(b), note that we are given that
\[ (x,y) = 0\implies (Ax,y) = 0. \]
Thus, writing \(Ax = \lambda x + x'\), we see that \((x',y)=0\) for all \(y\in\Span\{x\}^\perp\), so \((x',x')=0\), meaning \(x'=0\). We conclude that \(Ax = \lambda x\), where \(\lambda\) now depends on \(x\).
However, by linearity, it must be independent of \(x\), so that \(A = \lambda I\).
\end{proof}
\begin{lemma}\label{lemma:E-of-basis-sum-to-I}
	Let \(\{e_i\}_{i=1}^n\) be an orthonormal basis of \(\C^n\). Then
	\[ I = \sum_{i=1}^n E_{e_ie_i}. \]
\end{lemma}
\begin{proof}
Note that for all \(v\in \C^n\), we have \(E_{e_ie_i}v = (e_i,v) e_i\). In particular,
\[ \sum_{i=1}^n E_{e_ie_i} v = \sum_{i=1}^n (e_i,v)e_i = v \]
as desired.
\end{proof}

Here is the central trick we will use to deduce our result. It tells us that any subspace of \(\End(\C^n)\) satisfying a meagre amount of closure properties must contain all matrices.
\begin{lemma}\label{lemma:subspace-of-matrices-spans}
	Let \(V\subseteq\End(\C^n)\), \(n \geq 2\), be a linear subspace with the following properties:
	\begin{enumerate}[label=(\arabic*)]
	\item \(I\in V\).
	\item For some \(x,y\in\C^n\) of unit length such that \((x,y)=0\), we have \(E_{xy}\in V\) and \(E_{yx}\in V\).
	\item If \(A\in V\), then for all \(v,w\in\C^n\) such that \((v,w)=0\), we have \([A,E_{vw}]\in S\).
	\end{enumerate}
	Then \(V=\End(\C^n)\).
\end{lemma}
\begin{proof}
By Proposition \ref{prop:Exy-properties}, we have
\[ E_{vw}E_{wu} = E_{vu},\quad \text{and if }(u,v)=0, \quad E_{wu}E_{vw} = 0. \]
Extend the vectors \(\{x,y\}\) from (2) to a basis \(\{e_1,\ldots,e_n\}\), where \(x = e_1\) and \(y=e_2\). For simplicity, we will write \(E_{ij} := E_{e_ie_j}\). Using property (2) along with property (3) twice,
as long as \(k\not=\ell\) and at least \(k\not=2\) or \(\ell\not=1\), we have
\begin{align*}
	[ [E_{12}, E_{2\ell}], -E_{k1} ] &= [E_{12}E_{2{\ell}} - E_{2{\ell}}E_{12}, -E_{k1}] \\
	&= [E_{k1}, E_{1{\ell}} - E_{2{\ell}}E_{12}] \\
	&= E_{k1}E_{1{\ell}} - E_{k1}E_{2{\ell}}E_{12} - E_{1{\ell}}E_{k1} + E_{2{\ell}}E_{12}E_{k1} \\
	&= E_{k{\ell}} + E_{2{\ell}}E_{12}E_{k1} = E_{k\ell}\in V.
\end{align*}
Since \(E_{21}\in V\) by assumption, we have that \(E_{ij}\in V\) for all \(i\not= j\).

This leaves the ``diagonals'' \(E_{ii}\). To get \(E_{ii}\in V\), we use that, by the above, we have \(E_{ij}\in V\) for all \(i\not=j\). Therefore, we see, for \(i\not=k\), that \(E_{ki}\in V\).
Applying (3), we have
\[ [E_{ik},E_{ki}] = E_{ik}E_{ki}-E_{ki}E_{ik} = E_{ii} - E_{kk} \in V. \]
Therefore, fixing \(i\) and summing up over all \(k\not i\) we have
\[ (n-1)E_{ii} - \sum_{k\not= i} E_{kk} = (n-1)E_{ii} - (I - E_{kk}) = nE_{ii} + I \in V \]
by Lemma \ref{lemma:E-of-basis-sum-to-I}, which, using (1), means that \(E_{kk}\in V\).
\end{proof}

\begin{theorem}\label{thm:S^phi(Z)-classification}
	Let \(Z\in\End(\C^n)\). Then \(\frakS_\cap^\varphi(Z) = \{\lambda I \mid \lambda\in\C\}\) or \(\End(\C^n)\).
\end{theorem}
\begin{proof}
Suppose \(\frakS_\cap^\varphi(Z)\not=\{\lambda I\mid \lambda\in\C\}\). We use Corollary \ref{corollary:S^phi-local-closure-properties} to show that \(\frakS_\cap^\varphi(Z)\) satisfies the conditions of Lemma \ref{lemma:subspace-of-matrices-spans}.
Condition (1) is clear, and condition (3) follows by part (ii)(a) of Corollary \ref{corollary:S^phi-local-closure-properties}. It remains to show that \(\frakS_\cap^\varphi(Z)\) satisfies property (2). For this,
pick \(A\in\frakS_\cap^\varphi(Z)\) such that \(A \not= \lambda I\), and note that for all \((x_0,y_0) = 0\) we have
\[ E_{x_0y_0}AE_{x_0y_0} = a_{y_0x_0}E_{x_0y_0}\in\frakS_\cap^\varphi(Z). \]
Since \(A\not=\lambda I\), we know by (iv)(b) that there is some choice of \((x,y) = 0\) such that \(a_{yx}\not=0\), so that \(E_{xy}\in\frakS_\cap^{\varphi}\). Finally, let \(T_{xy}\) be the
operator given by
\[ T_{xy}x := y,\quad T_{xy}y := x,\quad \forall z\in\Span\{x,y\}^\perp,\, Tz = z. \]
Then \(T_{xy}^2 = I\), so
\[ T_{xy}E_{xy}T_{xy} = E_{yx} \in \frakS_\cap^\varphi(Z) \]
by Corollary \ref{corollary:S^phi-local-closure-properties}.
\end{proof}

\subsection{Existence Without a Modifier \& With Some Modifiers}\label{subsection:unbounded-unmodified-existence}
Now, Theorem \ref{thm:S^phi(Z)-classification} tells us that in finite dimensions, for general \(\varphi\), the space \(\frakS_\cap^\varphi(Z)\) has two possible forms. We now show that in a very general setting,
we can say exactly what \(\frakS_\cap(Z)\) is. Using the same exact trick as in Section \ref{subsection:bounded-modified-existence}, via the \(\preceq\) relation of Section \ref{subsection:preorder}, we will also be able to
extend this to a few more choices of \(\varphi\).

\begin{lemma}\label{lemma:subspace-with-closure-violates-invariance}
	Let \(X\) be a Hilbert space, and let \(V\subseteq X\) be a (not necessarily closed) linear subspace such that
	\begin{enumerate}[label=(\arabic*)]
	\item there are some orthonormal \(e,e'\in X\) such that \(E_{ee'}\in V\) and \(E_{e'e}\in V\),
	\item if \(A\in V\), then for all orthonormal \(v,w\in\C^n\), we have \([A,E_{vw}]\in S\), and
	\item if \(A\in V\), then for all operators \(T\in\sfB(X)\) such that \(T^2 = I\), we have \(TAT\in V\).
	\end{enumerate}
	Then, for any non-trivial linear subspace \(Y\subseteq X\), there is some \(A\in V\) such that \(AY\not\subseteq Y\).
\end{lemma}
\begin{proof}
By Zorn's lemma, we may extend \(\{e,e'\}\) to an orthonormal basis \(\{e_i\}_{i\in I}\) of \(X\), where \(I\) may be uncountable. Write \(E_{ij} := E_{e_ie_j}\). Following the argument of
Lemma \ref{lemma:subspace-of-matrices-spans}, (1) and (2) imply that \(E_{ij}\in V\) for all \(i\not=j\in I\). To get \(E_{ii}\in V\) for all \(i\in I\), define an operator \(T_{kj}\) for all \(k\not=j\in I\)
by
\[ T_{jk}e_j = e_k - e_j,\quad T_{jk}(e_k-e_j) = e_j,\quad \forall z\in\Span\{e_j,e_k-e_j\}^\perp,\, T_{jk}z = z. \]
Then \(T_{kj}^2 = I\), and we note that for \(j\not=i\)
\[ T_{ji} E_{ij} T_{ji}(e_i - e_j) = T_{ji}E_{ij}e_j = T_{ji}e_i = T_{ij}(e_i - e_j + e_j) = e_j + e_i - e_j = e_i, \]
and more generally,
\[ T_{ji}E_{ij}T_{ji} = E_{e_i(e_j-e_i)} \]
which is in \(V\) by (3). However, we also have
\[ E_{e_i(e_j-e_i)} = E_{ij} - E_{ii} \implies E_{ii} = E_{ij} - E_{e_i(e_j-e_i)} \in V \]
so that \(E_{ii}\in V\).

Now, for any \(y\in X\), there is some countable subset \(I_y\subseteq I\) for which we may write
\[ y = \sum_{i\in I_y} (e_i,y)e_i. \]
Since \(Y\) is non-trivial, we can find \(y\in Y\backslash\{0\}\) and some basis vector \(e_k\not\in Y\). We then have
\[ E_{e_ky} = \sum_{i\in I_y}(e_i,y)E_{ki}. \]
Choosing a bijection \(|\cdot|\!:I_x\iso\N\), for any \(K > 0\) we define
\[ E_{e_ky}^{(K)} := \sum_{\substack{i\in I_x \\ |i| < K}} (e_i,y)E_{ki} \]
and note that for all \(K>0\), we have \(E_{e_ky}^{(K)}\in V\) by the first half of the proof. Finally, since
\[ E_{e_ky}^{(K)} \to E_{e_ky}\quad \text{as }K\to\infty, \]
we may find some \(N > 0\) such that
\[ E_{e_ky}^{(N)}y \not= 0, \]
which, setting \(A := E^{(N)}_{e_ky}\), means we have some \(A\in V\) such that \(Ay = \lambda e_k\) for some \(\lambda\not=0\), and in particular, \(AY\not\subseteq Y\) since \(e_k\not\in Y\).
\end{proof}

\begin{theorem}\label{thm:S(singular)-is-cI}
	Let \(X\) be a Hilbert space, and let \(Z\in\sfB(X)\) be a singular operator. Then
	\[ \frakS_\cap(Z) = \{\lambda I \mid \lambda\in\C\}. \]
\end{theorem}
\begin{proof}
First consider the case when \(Z\) is non-zero. We will argue by contradiction: suppose the assertion were \emph{not} true. By Theorem \ref{thm:kernel-invariant-algebra}, we have that
\[  \{ \lambda I \mid \lambda\in\C \}\subseteq \frakS_\cap(Z)\subseteq\frakS_\cup(Z) \subseteq \frakS_{\ker}(Z) \]
and since \(Z\not=0\) and is singular, we know that \(\ker(Z)\) is a non-trivial subspace of \(X\). Therefore, we may try to employ Lemma \ref{lemma:subspace-with-closure-violates-invariance}.
However, conditions (2) and (3) follow trivially from Corollary \ref{corollary:S^phi-local-closure-properties}, and condition (1) follows by the same argument as in
Theorem \ref{thm:S^phi(Z)-classification}.

When \(Z=0\), note that for any \(x\in X\), we may consider the orthogonal projection \(P_x\) onto \(\Span\{x\}^\perp\), which has \(\ker(P_x) = \Span\{x\}\). For
any \(A\in\frakS_\cap(Z)\), we will then have that \(A\in\frakS_\cap(\varepsilon P_x)\subseteq \frakS_{\ker}(\varepsilon P_x) = \frakS_{\ker}(P_x)\) for some sufficiently small \(\varepsilon > 0\) by Theorem \ref{thm:S(Z)-contained-in-cup-cap-ker},
and in particular, \(Ax = \lambda_x x\) for all \(x\). A simple argument by linearity then shows that \(A = \lambda I\).
\end{proof}

From the above, we now know exactly how \(\frakS_\cap(-)\) looks in Hilbert spaces. In finite dimensions, therefore, we can then apply the general Corollary \ref{corollary:phi-simeq-psi-implies-same-S}
to deduce \(\frakS_\cap^\varphi(-)\) looks for some simple but interesting choices of \(\varphi\).

\begin{theorem}\label{thm:hadamard-sup-version-J-class}
	Let \(\bbJ\) be as given in Theorem \ref{thm:hadamard-local-infimum-J-theorem}, and let \(Z\in\End(\C^n)\). If \(\varphi\in\bbJ\), then
	\[ \frakS_\cap^\varphi(Z) = \frakS_\cap(Z). \]
	In particular, suppose \(Z\) is singular. Then \(\frakS_\cap^\varphi(Z) = \{\lambda I \mid \lambda\in\C\}\).
\end{theorem}
\begin{proof}
Combine Theorem \ref{thm:S(singular)-is-cI} with Corollary \ref{corollary:phi-simeq-psi-implies-same-S} and Lemma \ref{lemma:J-is-equivalent-to-id}.
\end{proof}

\begin{remark}
	Like with \(\frakS_\cup(-)\) and \(\frakS_\cup^*(-)\), one may define a dual version, \(\frakS_\cap^*(-)\), of \(\frakS_\cap(-)\). This has the peculiar property that it is very
	easy to show that \(\frakS_\cap^*(0) = \sfB(X)\) in any Banach space \(X\), using dualized versions of Theorems \ref{thm:kernel-invariant-algebra} \& \ref{thm:S(Z)-contained-in-cup-cap-ker}.
\end{remark}

\clearpage
\section{A Few Conjectures \& A Question}
We end this paper with a few conjectures on the objects we have considered. The first is that Theorem \ref{thm:S^phi(Z)-classification} can be extended to all Hilbert spaces.
\begin{conjecture}
	Let \(X\) be a Hilbert space, \(Y\) a Banach space, and let \(\varphi\!:\sfB(X)\to Y\) be a bounded linear map. Then, for all singular \(Z\in\sfB(X)\), either \(\frakS_\cap^\varphi(Z) = \{\lambda I\mid \lambda\in\C\}\)
	or \(\frakS^\varphi(Z) = \sfB(X)\).
\end{conjecture}

For our second conjecture, we would like to say that the families \(\frakS_\cap^\varphi(-)\) and \(\frakS_\cup^\varphi(-)\) are tightly related to each other. Since we have a necessary and sufficient condition
for all \(\varphi\) that satisfy \(\frakS_\cup^\varphi(-) = \frakS_\cup(-)\) (given by Theorem \ref{thm:S_V^phi=S_V-implies-phi(2nd-ord-nilp)-nonzero}), ideally we would like those same criteria
to determine the situation for \(\frakS_\cap^\varphi(-)\).

\begin{conjecture}
	Let \(X\) be a Hilbert space, \(Y\) a Banach space, and let \(\varphi\!:\sfB(X)\to Y\) be a bounded linear map. Then the following conditions are equivalent.
	\begin{enumerate}[label=(\arabic*)]
	\item For all \(Z\in\sfB(X)\) of generalized index zero, \(\frakS_\cap^\varphi(Z) = \frakS_\cap(Z)\).
	\item For all \(Z\in\sfB(X)\) of generalized index zero, \(\frakS_\cup^\varphi(Z) = \frakS_\cup(Z)\).
	\end{enumerate}
\end{conjecture}

Our third conjecture is harder to motivate. Essentially, in the process of proving the results of Section \ref{section:existence-of-unbounded-paths}, several approaches came up. One
of them involved doing a computation of the kernel of a particular map, which conjures forth ideas that results in Section \ref{section:existence-of-unbounded-paths} should be
approachable with the concepts developed in Section \ref{section:existence-of-bounded-paths}.

\begin{conjecture}
	Let \(V\) be a Banach algebra, \(Y\) a Banach space, and \(\varphi\!:V\to Y\) a bounded linear map. Then
	\[ \forall z\in V,\quad \frakS_\cap^\varphi(z) = \bigcup\left\{ \bigcap_{z'\in U}\frakS^\varphi_\cup(z') : U\subseteq V\text{ open and bounded, } z\in U \right\}. \]
\end{conjecture}

In Section \ref{subsection:bounded-unmodified-existence}, we crucially make use of the condition that our operators are generalized index zero in order to prove our results (in particular, to write
our base space as a direct sum so we can produce good paths and do computations).

\begin{question}
	Does Corollary \ref{corollary:hilbert-space-S_V=S_ker} hold even for operators \(Z\) that are not of generalized index zero? That is, is it true that
	as long as \(Z\in\sfB(X)\) can be approximated by invertible operators, then \(\frakS_\cup(Z) = \frakS_{\ker}(Z)\)?
\end{question}

\clearpage
\begin{appendices}
\section{Variant of the Proof of the Lemma \ref{lemma:J-is-equivalent-to-id}}\label{section:appendix-variant-of-modified}
We provide here a variant of the proof of Lemma \ref{lemma:J-is-equivalent-to-id}, which chronologically came first and is due to the first author.
\begin{lemma}\label{lemma:inf-J-bounded-paths-are-inf-bounded}
	Let \(Z,A\in\End(\C^n)\), let \(U_k\to Z\) be a sequence of invertible matrices converging to \(Z\), and let \(J=\1-I\). Then
	\[ \lim_{k\to \infty}\| J * (U_kAU_k^{-1})\| < \infty \iff \lim_{k\to\infty} \|U_kAU_k^{-1}\| < \infty. \]
\end{lemma}
\begin{proof}
The \((\Leftarrow)\) direction is clear. To prove the \((\To)\) direction, set \(B_k := U_kAU_k^{-1}\). Then, by assumption,
we have that \(\lim_{k\to\infty}\|J*B_k\|<\infty\), so that the non-diagonal entries of \(B_k\) remain bounded as \(k\to\infty\).
Now we have the following facts:
\begin{enumerate}[label=(\roman*)]
\item Since each \(B_k\) is a conjugate of \(A\), they all have the same eigenvalues as \(A\).
\item By the Gershgorin circle theorem, the eigenvalues of \(A\) lie within some bounded disks centered on the diagonal entries of \(A\).
\item On the other hand, also by the Gershgorin circle theorem, the eigenvalues of \(B_k\) must lie within some \emph{bounded radius, which is independent of} \(k\),
of the diagonal entries of \(B_k\) because, by assumption, the non-diagonal entries stay bounded as \(k\to\infty\). That is,
the supremum of the radii of the Gershgorin circles of the \(B_k\) is some finite number.
\end{enumerate}
If some diagonal entry of \(B_k\) exploded as \(k\to\infty\), this would imply that one of the eigenvalues of \(B_k\), i.e.\ one of the eigenvalues of \(A\),
lies outside the Gershgorin circles of \(A\). This is impossible, hence the diagonal entries must necessarily also stay bounded, and therefore
\(\|B_k\| = \|U_kAU_k^{-1}\|\) stays bounded as \(k\to\infty\).
\end{proof}

\section{Notation}\label{section:appendix-notation}

\begin{tabular}{l p{11.25cm}}\label{table-of-notation}
	Symbol & Meaning \\ \hline
	\(H*A\) & The Hadamard product of \(H\) with \(A\).\\
	\(\varphi\preceq\psi\), \(\varphi\simeq\psi\) & See Notation \ref{notation:preorder-relation}.\\
	\(\sfB(X)\) & The Banach algebra of bounded linear operators \(X\to X\).\\
	\(\frakC_{-1}(z)\) &  The set of elements appearing as the \((-1)\)th coefficient of the inverse of a good path converging to \(z\). See Definition \ref{definition:frakC}.\\
	\(\frakC'(Z)\) & The set of bounded linear operators \(C\) such that \(\img(C) = \ker(Z)\) and \(\ker(C)=\img(Z)\). See Definition \ref{definition:frakC'}. \\
	\(\End(\C^n)\) & The (Banach) algebra of linear maps \(\C^n\to\C^n\). After a choice of basis, can be identified with the algebra \(\C^{n\times n}\) of \(n\times n\) matrices.\\
	\(\frakN^\varphi_\cap\), \(\frakN_\cup^\varphi\) & See Definition \ref{definition:frakN}.\\
	\(\frakS^\varphi_\cap(-)\), \(\frakS^\varphi_\cup(-)\) & See Definition \ref{definition:frakS}.\\
	\(\frakS_{\ker}(Z)\) & The algebra of bounded linear operators keeping \(\ker{Z}\) invariant. See Definition \ref{def:kernel-invariant-set}.\\
	\(\frakS_c(-)\), \(\frakS_c^\varphi(-)\) & See Definition \ref{definition:frakS_c(z)} and Definition \ref{definition:frakS_c^phi(z)}.\\
	\(\frakS_{\pw}(z;u_k)\) & See Definition \ref{definition:frakS_pw}.\\
	\(\frakS_\frakF(F^\bullet V)\) & See Definition \ref{definition:frakS-filtration}.\\
	\(V^\times\) & The group of invertible elements in an algebra \(V\).
\end{tabular}

\end{appendices}

\clearpage
\phantomsection
\addcontentsline{toc}{section}{References}
\printbibliography

\end{document}